\def\version{29/03/2017 version 8
\hfill\href{http://arxiv.org/abs/1208.1868}{arXiv:1208.1868}
}

\documentclass[11pt,a4paper]{amsart}
\usepackage{fullpage}

\usepackage{amssymb,amscd}

\usepackage[all]{xy}
\newdir{ >}{{}*!/-8pt/@{>}}
\usepackage{mathrsfs}
\usepackage[numeric,lite]{amsrefs}

\usepackage{pigpen}
\def\PO{\text{\pigpenfont R}}

\renewcommand{\thefootnote}{\fnsymbol{footnote}}
\long\def\symbolfootnote[#1]#2{\begingroup%
\def\thefootnote{\fnsymbol{footnote}}\footnote[#1]{#2}\endgroup}

\usepackage{hyperref}

\newtheorem{thm}{Theorem}[section]
\newtheorem{lem}[thm]{Lemma}
\newtheorem{prop}[thm]{Proposition}
\newtheorem{cor}[thm]{Corollary}

\theoremstyle{definition}
\newtheorem{rem}[thm]{Remark}
\newtheorem{defn}[thm]{Definition}
\newtheorem{examp}[thm]{Example}

\numberwithin{equation}{section}
\numberwithin{figure}{section}

\def\ie{\emph{i.e.}}

\def\:{\colon}
\def\.{\cdot}
\def\o{\circ}
\def\<{\left\langle}
\def\>{\right\rangle}
\def\({\left(}
\def\){\right)}
\def\ph#1{\phantom{#1}}
\def\epsilon{\varepsilon}
\def\phi{\varphi}

\def\leq{\leqslant}
\def\geq{\geqslant}
\def\lla{\longleftarrow}

\def\Lra{\Longrightarrow}
\def\ra{\rightarrow}
\def\bar#1{\overline{#1}}

\def\tilde#1{\widetilde{#1}}
\def\iso{\cong}

\DeclareMathOperator{\id}{id}
\DeclareMathOperator{\Id}{Id}

\DeclareMathOperator{\im}{im}
\DeclareMathOperator{\pinch}{pi}

\def\CP{\mathbb{C}\mathrm{P}}
\def\CPi{\CP^\infty}

\def\F{\mathbb{F}}

\def\H{\mathbb{H}}

\def\k{\Bbbk}

\def\Q{\mathbb{Q}}

\def\RP{\mathbb{R}\mathrm{P}}
\def\RPi{\RP^\infty}
\def\Z{\mathbb{Z}}

\DeclareMathOperator{\Hom}{Hom}

\DeclareMathOperator{\ord}{ord}

\DeclareMathOperator{\TAQ}{TAQ}
\DeclareMathOperator{\Tor}{Tor}
\DeclareMathOperator{\Tr}{Tr}

\def\del{\partial}

\DeclareMathOperator{\dlQ}{\mathrm{Q}}

\def\Sq{\mathrm{Sq}}
\def\SQ{\mathrm{SQ}}
\DeclareMathOperator{\exc}{excess}
\DeclareMathOperator{\len}{len}

\def\undervee#1{\underset{#1}\vee}
\def\Kriz{{K\v{r}\'{i}\v{z}}}
\def\Einfty{$\mathcal{E}_\infty$ }
\title[Calculating with topological Andr\'e-Quillen theory, I]
{Calculating with topological Andr\'e-Quillen theory, I: \\
Homotopical properties of universal derivations and free
commutative \boldmath$S$-algebras}
\author{Andrew Baker}
\date{\version}
\address{
School of Mathematics \& Statistics, University of Glasgow,
Glasgow G12 8QW, Scotland.}
\email{a.baker@maths.gla.ac.uk}
\urladdr{http://www.maths.gla.ac.uk/$\sim$ajb}
\thanks{Part of this work was carried out in the period 2007--8
when the author was supported by a YFF Norwegian Research Council
grant while at the University of Oslo, and later while receiving
an EPSRC research grant. The author would like to thank Maria
Basterra, Marcel B\"okstedt, Bob Bruner, Helen Gilmour, Nick
Kuhn, Tyler Lawson, Mike Mandell, Peter May, Birgit Richter,
Constanze Roitzheim, Markus Szymik and John Rognes for helpful
comments and encouragement over many years, and especially
Philipp Reinhard for detailed comments and also supplying the
proof in Appendix~\ref{sec:Missingpf}. \\
This research was supported by funding from RCUK}

\begin{document}

\begin{abstract}
We adopt the viewpoint that topological And\'e-Quillen
theory for commutative $S$-algebras should provide usable
(co)homology theories for doing calculations in the sense
traditional within Algebraic Topology. Our main emphasis
is on homotopical properties of universal derivations,
especially their behaviour in multiplicative homology
theories. There are algebraic derivation properties, but
also deeper properties arising from the homotopical
structure of the free algebra functor $\mathbb{P}_R$ and
its relationship with extended powers of spectra. In the
connective case in ordinary $\bmod{\,p}$ homology, this
leads to useful formulae involving Dyer-Lashof operations
in the homology of commutative $S$-algebras. Although
many of our results could no doubt be obtained using
stabilisation, our approach seems more direct. We also
discuss a reduced free algebra functor $\tilde{\mathbb{P}}_R$.
\end{abstract}

\maketitle

%\tableofcontents

\section*{Introduction}

Topological Andr\'e-Quillen homology and cohomology theories
for commutative $S$-algebras were introduced by Maria Basterra,
building on ideas of Igor \Kriz{} as well as algebraic Andr\'e-Quillen
theory. Subsequent work, both individually and jointly in various
combinations, by Basterra, Gilmour, Goerss, Hopkins, Kuhn, Lazarev,
Mandell, McCarthy, Minasian, Reinhard, Richter, Robinson, Whitehouse
as well as the present author, has laid out the basic structure
and provided key relationships with other areas.

In this work we continue to adopt the viewpoint of~\cite{BGRtaq},
regarding $\TAQ$ as providing usable (co)homology theories for
doing calculations in the sense traditional within Algebraic
Topology.

Our main emphasis is on homotopical properties of universal
derivations, especially their behaviour in multiplicative
homology theories. As the name suggests, there are algebraic
derivation properties, but also deeper properties arising out
of the homotopical structure of the free algebra functor and
its relationship with extended powers of spectra. In the
connective case and in ordinary mod~$p$ homology, this leads
to useful formulae involving Dyer-Lashof operations in the
homology of commutative $S$-algebras. It seems likely that
many of these results are obtainable using stabilisation, but
our approach seems more direct. We remark that work of Mike
Mandell~\cite{MM:TAQ} suggests that it might be more natural
to replace commutative $S$-algebras by algebras in $\mathcal{M}_S$
over his operad $\mathcal{G}$ and work with those. Related
results on the homology of the free commutative $S$-algebra
functor also appear in work of Nick Kuhn \& Jason
McCarty~\cites{NJK:Transfers,NJK&JBMcC:HomLoopSpces}.

We also discuss a \emph{reduced free algebra} functor
$\tilde{\mathbb{P}}_R$ which we learnt of from Tyler Lawson.
This takes as input $R$-modules under a fixed cofibrant
replacement for the $R$-module~$R$ and gives rise to a Quillen
adjunction.
\[
\xymatrix{
{\mathscr{C}_R} \ar@/_8pt/[rr]_{\tilde{\mathbb{U}}}
 && {S^0_R/\mathscr{M}_R} \ar@/_8pt/[ll]_{\tilde{\mathbb{P}}_R}
}
\]
We will use this in a sequel to study spectral sequences related
to those studied by Maria Basterra~\cite{MBtaq}*{section~5} and
Haynes Miller~\cite{HRM:SS}.

We give some sample calculations, but our main concern is with
laying the groundwork for future applications.

In two brief appendices we supply a proof of a basic result,
an adjunction result, and some formulae for calculating
Dyer-Lashof operations.

\subsection*{Notation, etc}

When working over a fixed commutative ground ring $\k$ such
as $\F_p$, we often write $\otimes$ for $\otimes_{\k}$,
$\Hom$ for $\Hom_{\k}$, etc.

\section{Recollections on topological Andr\'e-Quillen theory}
\label{sec:Recollections}

We will assume the reader is familiar with Basterra's foundational
paper~\cite{MBtaq} and the further development of its ideas
in~\cite{BGRtaq}. All of this is founded on the notions of $S$-modules
and commutative $S$-algebras of~\cite{EKMM}. We briefly spell
out some of the main ingredients.

If $R$ is a commutative $S$-algebra, then its category of (left)
$R$-modules $\mathscr{M}_R$ is a model category and the category
of commutative $R$-algebras $\mathscr{C}_R$ consists of the
commutative monoids in $\mathscr{M}_R$ with monoidal morphisms.
There is a free $R$-algebra functor
$\mathbb{P}_R\:\mathscr{M}_R\to\mathscr{C}_R$ left adjoint to
the forgetful functor $\mathbb{U}\:\mathscr{C}_R\to\mathscr{M}_R$,
and this pair gives a Quillen adjunction. We denote the derived
(or homotopy) categories of these model categories by
$\bar{h}\mathscr{M}_R$ and $\bar{h}\mathscr{C}_R$.

For every pair of commutative $R$-algebras $A\to B$, their
\emph{cotangent complex} is a $B$-module $\Omega_A(B)$ which
is well defined up to isomorphism in $\bar{h}\mathscr{M}_B$.
This comes with a canonical morphism in $\bar{h}\mathscr{M}_A$,
the \emph{universal derivation}
\[
\delta_{(B,A)}\:B\to\Omega_A(B),
\]
characterised by a natural isomorphism
\[
\bar{h}\mathscr{C}_A/B(B, B\vee X) \iso \bar{h}\mathscr{M}_B(\Omega_A(B),X),
\]
where $X\in\mathscr{M}_B$ and $B\vee X$ denotes the \emph{square
zero extension} of~$B$ by~$X$ viewed as an $A$-algebra over~$B$.

\emph{Topological Andr\'e-Quillen homology} and \emph{cohomology}
with coefficients in a $B$-module $M$ are defined by
\begin{align*}
\TAQ_*(B,A;M) &= \pi_*(M\wedge_B \Omega_A(B)), \\
\TAQ^*(B,A;M) &= \pi_{-*}(F_B(\Omega_A(B),M))
               = \bar{h}\mathscr{M}_B(\Omega_A(B),M)^*,
\end{align*}
where
\[
\bar{h}\mathscr{M}_B(\Omega_A(B),X)^n
    = \bar{h}\mathscr{M}_B(\Omega_A(B),\Sigma^n X).
\]
When $E$ is a (unital) $B$ ring spectrum, the composition
\begin{equation}\label{eq:TAQ-Hurewiczhom}
\xymatrix{
\pi_*(B) \ar[r]\ar@/^19pt/[rrrr]^{\theta}
  & E_*(B) \ar[rr]_{(\delta_{(B,A)})_*}
  && E_*(\Omega_A(B))\ar[r]_(.35){\theta'}
                & E^B_*(\Omega_A(B)) = \TAQ_*(B,A;E)
}
\end{equation}
is the \emph{$\TAQ$-Hurewicz homomorphism}.

In \cite{BGRtaq} we showed how this could be interpreted
as a cellular theory for cellular commutative $R$-algebras.
A key ingredient was the basic observation that for a
cofibrant $R$-module $X$, cotangent complex of the free
$R$-algebra $\mathbb{P}_RX$ is
\begin{equation}\label{eq:OmegaPX}
\Omega_R(\mathbb{P}_RX) \iso \mathbb{P}_RX\wedge_R X,
\end{equation}
see~\cite{BGRtaq}*{proposition~1.6} for example. For
completeness we give a proof of this in
Appendix~\ref{sec:Missingpf}, and discuss the universal
derivation for $\mathbb{P}_RX$ in Section~\ref{sec:PX}.

In \cite{BGRtaq} we developed the theory of connective
$p$-local commutative $S$-algebras along the lines
of~\cite{AJB-JPM} for spectra, making crucial use of
$\TAQ$ with coefficients in $H\F_p$. In both of those
works, one important outcome was the ability to detect
\emph{minimal atomic} objects using the vanishing of
the appropriate Hurewicz homomorphism in positive
degrees.

\section{Homotopical properties of universal derivations}
\label{sec:UnivDer}

Let $A$ be a commutative $S$-algebra. As pointed out
in~\cite{AL:Glasgow}, for a commutative $A$-algebra~$B$,
the universal derivation $\delta_{(B,A)}\:B\to\Omega_A(B)$
is a homotopy derivation in the sense of the following
discussion.

Suppose that $R$ is a commutative $S$-algebra, let~$E$
be an~$R$ ring spectrum and let $M$ be a left $E$-module.
We remind the reader this means that there are morphisms
$\eta\:R\to E$, $\phi\:E\wedge_R E\to E$ and
$\mu\:E\wedge_R M\to M$ in $\bar{h}\mathscr{M}_R$ which
satisfy appropriate associativity and unital conditions.
\begin{defn}\label{defn:HtpyDeriv}
A morphism $\del\:E\to M$ in $\bar{h}\mathscr{M}_R$ is
a \emph{homotopy derivation} if the following diagram
in $\bar{h}\mathscr{M}_R$ commutes.
\begin{equation}\label{eq:hder}
\xymatrix{
&E\wedge_R E\ar[rr]^{\phi}\ar[dl]_{I\wedge\del\vee\del\wedge I}&&E\ar[dr]^{\del}& \\
E\wedge_R M\vee M\wedge_R E\ar[dr]_{I\vee\mathrm{switch}} &&&& M \\
&E\wedge_R M\vee E\wedge_R M\ar[rr]^(.6){\mathrm{fold}} &&E\wedge_R M\ar[ur]_{\mu}&
}
\end{equation}
\end{defn}

\smallskip
Now let $A$ be a commutative $S$-algebra and let $B$ be
a commutative $A$-algebra. Following the remarks at end
of~\cite{AL:Glasgow}*{section~2}, we recall that the
universal derivation $\delta_{(B,A)}\:B\to\Omega_A(B)$
is a morphism in  the derived category of $A$-modules
$\bar{h}\mathscr{M}_{A}$ which is also a homotopy
derivation in the sense that the following diagram
commutes in $\bar{h}\mathscr{M}_{A}$
\begin{equation}\label{eq:Deriv}
\xymatrix{
B\wedge_AB\ar[rr]^(.57){\mathrm{prod}}
\ar[dd]_{I\wedge\delta_{(B,A)}+\mathrm{switch}\circ(\delta_{(B,A)}\wedge I)}
     && B\ar[dd]^{\delta_{(B,A)}} \\
     && \\
B\wedge_A\Omega_A(B)\ar[rr]^(.57){\mathrm{mult}}  && \Omega_A(B)
}
\end{equation}
where elements of $\bar{h}\mathscr{M}_{A}(X,Y)$ are added
in the usual way.

Now suppose that $E$ is a commutative $B$ ring spectrum;
this implies that $E$ is a $B$-module and there is a unit
morphism of $B$ ring spectra $B\to E$ in $\bar{h}\mathscr{M}_B$.
Then on smashing with copies of $E$,~\eqref{eq:Deriv}
gives another commutative diagram
\begin{equation}\label{eq:Deriv-E}
\xymatrix{
E\wedge_AB\wedge_AE\wedge_AB\ar[rr]^{\mathrm{switch}}
  \ar[dd]^{\substack{ I\wedge I\wedge I\wedge\delta_{(B,A)} \ph{abcabcabcabcabc}\\
  \ph{123} + \mathrm{switch}\circ(I\wedge \delta_{(B,A)}\wedge I\wedge I)} }
 && E\wedge_A E\wedge_A  B\wedge_AB\ar[rr]^{\mathrm{prod}}
 \ar[dd]^{\substack{I\wedge I\wedge I\wedge\delta_{(B,A)} \ph{abcabcabcabcabc} \\
    \ph{123} + I\wedge I\wedge\mathrm{switch}\circ(\delta_{(B,A)}\wedge I)}}
 && E\wedge_AB\ar[dd]^{I\wedge\delta_{(B,A)}}  \\
 &&  &&  \\
E\wedge_AB\wedge_AE\wedge_A\Omega_A(B)\ar[rr]^{\mathrm{switch}}
 && E\wedge_A E\wedge_AB\wedge_A\Omega_A(B)
                \ar[rr]^(.57){\mathrm{prod}\wedge\mathrm{mult}}
 && E\wedge_A\Omega_A(B)
}
\end{equation}
which shows that the commutative $E_*$-algebra
$E^A_*B=\pi_*(E\wedge_A B)$ admits the $E_*$-module
homomorphism
\[
(\delta_{(B,A)})_*\:E^A_*B\to E^A_*\Omega_A(B).
\]
Of course $E^A_*\Omega_A(B)$ is also a left $E^A_*B$-module
since $\Omega_A(B)$ is a left $B$-module.
Composing $(\delta_{(B,A)})_*$ with the natural homomorphism
$E^A_*\Omega_A(B)\to E^B_*\Omega_A(B)$, we obtain an $E_*$-module
homomorphism
\[
\Delta_{(B,A)}\:E^A_*B \to E^A_*\Omega_A(B)\to E^B_*\Omega_A(B).
\]
We also have an augmentation $\epsilon\:E^A_*B\to E_*$ induced
by applying $\pi_*(-)$ to the evident composition
\[
E\wedge_AB \to E\wedge_AE \to E.
\]
Clearly $\epsilon$ is a morphism of $E_*$-algebras.
\begin{lem}\label{lem:Deriv-E}
$(\delta_{(B,A)})_*$ and $\Delta_{(B,A)}$ are $E_*$-derivations,
so for $u,v\in E^A_*B$,
\begin{align*}
(\delta_{(B,A)})_*(uv) &=
u(\delta_{(B,A)})_*(v)\pm v(\delta_{(B,A)})_*(u), \\
\Delta_{(B,A)}(uv) &=
\epsilon(u)\Delta_{(B,A)}(v)\pm\epsilon(v)\Delta_{(B,A)}(u),
\end{align*}
where the signs are determined from the degrees of $u,v$
with the usual sign convention. In particular, if
$u,v\in\ker\epsilon\:E^A_*B\to E_*$, then
\[
\Delta_{(B,A)}(uv) = 0,
\]
so $\Delta_{(B,A)}$ annihilates non-trivial products.
\end{lem}
\begin{proof}
This involves diagram chasing using the definitions.
\end{proof}
We will often write $\delta$ and $\Delta$ for $\delta_{(B,A)}$
and $\Delta_{(B,A)}$ when $(B,A)$ is clear from the context.

\section{The free commutative algebra functor}\label{sec:PX}

For a $R$-module $X$ there is a free commutative $R$-algebra
\[
\mathbb{P}_RX = \bigvee_{j\geq 0} X^{(j)}/\Sigma_j.
\]
When $R=S$ or a localisation of $S$, we will set $\mathbb{P}=\mathbb{P}_S$.

If $X$ is cofibrant as an $R$-module then $\mathbb{P}_RX$ is
cofibrant as an commutative $R$-algebra. The functor $\mathbb{P}_R$
is left adjoint to the forgetful functor
$\mathbb{U}\:\mathscr{C}_R\to\mathscr{M}_R$, so for $A\in\mathscr{C}_R$,
\[
\mathscr{C}_R(\mathbb{P}_R(-),A) \iso \mathscr{M}_R(-,A),
\]
where $A=\mathbb{U}A$ is regarded as an $R$-module. In
fact,
\begin{equation}\label{eq:QuillenAdj-C-M}
\xymatrix{
{\mathscr{C}_R} \ar@/_8pt/[rr]_{\mathbb{U}}
      && {\mathscr{M}_R} \ar@/_8pt/[ll]_{\mathbb{P}_R}
}
\end{equation}
is a Quillen adjunction~\cite{EKMM}.

As it is a left adjoint, $\mathbb{P}_R$ preserves colimits,
including pushouts. As cell and CW $R$-modules are defined
as iterated pushouts, applying $\mathbb{P}_R$ to the skeleta
leads to cell or CW skeleta. To make this explicit, suppose
that $X$ is an $R$-module with CW skeleta $X^{[n]}$, and
attaching maps
\[
j_n\:\bigvee_i S^n_R \to X^{[n]},
\]
where $S^n_R=\mathbb{F}_RS^n$ is the cofibrant model for the
sphere spectrum in $\mathscr{M}_R$. Setting $D^n_R=\mathbb{F}_RD^n$,
the $(n+1)$-skeleton $X^{[n+1]}$ is defined by the pushout
diagram
\[
\xymatrix{
\bigvee_i S^n_R\ar@{}[dr]|{\PO}
     \ar[r]\ar[d] & X^{[n]}\ar[d] \\
\bigvee_i D^{n+1}_R\ar[r] & X^{[n+1]}
}
\]
which induces the pushout diagram
\[
\xymatrix{
\mathbb{P}_R(\bigvee_i S^n_R)
 \ar@{}[dr]|{\PO}
 \ar[r]\ar[d]
          & \mathbb{P}_RX^{[n]}\ar[d] \\
\mathbb{P}_R(\bigvee_i D^{n+1}_R)\ar[r] & \mathbb{P}_RX^{[n+1]}
}
\]
in $\mathscr{C}_R$. So we obtain a CW filtration
on $\mathbb{P}_RX$ with $n$-skeleton
\[
\mathbb{P}_R^{\langle n\rangle}X =
(\mathbb{P}_RX)^{\langle n\rangle} = \mathbb{P}_R(X^{[n]}).
\]

Now we discuss the cotangent complex and universal derivation
for free algebras. Recalling~\eqref{eq:OmegaPX}, we know that
in the homotopy category $\bar{h}\mathscr{M}_{\mathbb{P}_RX}$,
\[
\Omega_R(\mathbb{P}_RX) \iso \mathbb{P}_RX \wedge_R X,
\]
and the universal derivation
\[
\delta_{(\mathbb{P}_RX,R)}
\in\bar{h}\mathscr{M}_R(\mathbb{P}_RX,\Omega_R(\mathbb{P}_RX))
=\bar{h}\mathscr{M}_R(\mathbb{P}_RX,\mathbb{P}_RX\wedge_R X)
\]
has the homotopy derivation property shown in the homotopy
commutative diagram~\eqref{eq:hder}. Furthermore,
$\delta_{(\mathbb{P}_RX,R)}$ corresponds to the inclusion
$X\to\mathbb{P}_RX\wedge_R X$ under the sequence of
isomorphisms
\begin{align}
\bar{h}\mathscr{C}_R/\mathbb{P}_RX(\mathbb{P}_RX,\mathbb{P}_RX\vee\Omega_R(\mathbb{P}_RX))
&\iso
\bar{h}\mathscr{M}_{\mathbb{P}_RX}(\Omega_R(\mathbb{P}_RX),\Omega_R(\mathbb{P}_RX))
                                           \notag \\
&\iso
\bar{h}\mathscr{M}_{\mathbb{P}_RX}(\mathbb{P}_RX\wedge_R X,\mathbb{P}_RX\wedge_R X)
                                          \notag \\
&\iso
\bar{h}\mathscr{M}_R(X,\mathbb{P}_RX\wedge_R X).
\label{eq:univder-modulemap}
\end{align}
We will describe $\delta_{(\mathbb{P}_RX,R)}$ as a morphism
in the homotopy category $\bar{h}\mathscr{M}_R$ using this
identification. Nick Kuhn has pointed out that~\cite{NJK:Transfers}
gives a closely related analysis of extended powers, and
proves far more about their coproduct structure induced
by the pinch map.

Suppose that $X'$ is a second copy of $X$. A representative
for the homotopy class of the pinch map
$\pinch\:X\to X\vee X'$ induces morphisms of
commutative $R$-algebras
\[
\xymatrix@C=0.5cm@R=1.0cm{
&& \mathbb{P}_R X\ar[dll]_{\mathbb{P}_R\pinch}\ar[drr] && \\
\mathbb{P}_R(X\vee X')\ar[dr]_(.45){\iso}
&& && \mathbb{P}_R X \vee \mathbb{P}_R X\wedge_R X' \\
 &\mathbb{P}_R X\wedge_R\mathbb{P}_R X'\ar[rr]
 && \mathbb{P}_R X\wedge_R (R\vee X')\ar[ur]_{\iso} &
}
\]
where $R\vee X'$ and $\mathbb{P}_R X\vee\mathbb{P}_R X\wedge_R X'$
are square zero extensions of $R$ and $\mathbb{P}_RX$ respectively,
and the horizontal morphism kills the wedge summands $(X')^{(r)}/\Sigma_r$
with $r\geq2$. Restricting to the summand $X$ in $\mathbb{P}_R X$
we obtain the pinch map~$\pinch$, and then on applying the
isomorphism of~\eqref{eq:univder-modulemap} we find that
the resulting composition
\[
\xymatrix{
\mathbb{P}_R X \ar[r]\ar@/^17pt/[rr]^\delta
 & \mathbb{P}_R X \vee \mathbb{P}_R X\wedge_R X' \ar[r]
 & \mathbb{P}_R X\wedge_R X'
}
\]
agrees with the universal derivation $\delta_{(\mathbb{P}_RX,R)}$.

In the homotopy category of $S$-modules, $\delta_{(\mathbb{P}_RX,R)}$
is equivalent to a coproduct of maps
\[
\delta_{(\mathbb{P}_RX,R),n}\:
E\Sigma_n\ltimes_{\Sigma_n}X^{(n)}
 \to (E\Sigma_{n-1}\ltimes_{\Sigma_{n-1}}X^{(n-1)})\wedge X,
\]
where we have of course identified $X'$ with $X$. In fact
these are the transfer maps $\tau_{n-1,1}$
of~\cite{LNM1176}*{definition~II.1.4}, \ie,
\begin{equation}\label{eq:delta-cpt}
\delta_{(\mathbb{P}_RX,R),n}=\tau_{n-1,1}\:
E\Sigma_n\ltimes_{\Sigma_n}X^{(n)}
 \to (E\Sigma_{n-1}\ltimes_{\Sigma_{n-1}}X^{(n-1)})\wedge X.
\end{equation}
The derivation property of $\delta_{(\mathbb{P}X,S)}$ is
just a consequence of the commutativity of the
diagram~\eqref{eq:DoubCoset} below. We will give a brief
explanation of this.

For detailed accounts of the stable homotopy theory involved,
see~\cites{LNM1176,LNM1213,JPM:Silverbook}. We remark that
in~\cite{LNM1176}*{chapter~II, p.~24}, the pinch map is
referred to as the `diagonal' since in the stable category
finite products and coproducts coincide. To ease notation
and exposition, we take $R=S$ and set $\mathbb{P}=\mathbb{P}_S$;
however the general case is similar.

Let $E\Sigma_{m+n}$ be a free contractible $\Sigma_{m+n}$-space,
and let~$Y$ be a $\Sigma_{m+n}$-spectrum for $m,n\geq1$; we
are interested in the case where $Y=X^{(m+n)}$, the $(m+n)$-th
smash power of~$X$. The equivariant half smash product
$E\Sigma_{m+n}\ltimes Y$ is a free $\Sigma_{m+n}$-spectrum,
and the evident inclusions of subgroups
\[
\xymatrix{
\Sigma_m\times\Sigma_n\ar@{^{(}->}[r]
         & \Sigma_{m+n}
         & \ar@{_{(}->}[l]\Sigma_{m+n-1}
}
\]
induce morphisms of spectra
\[
E\Sigma_{m+n}\ltimes_{\Sigma_m\times\Sigma_n} Y
    \to E\Sigma_{m+n}\ltimes_{\Sigma_{m+n}} Y
    \lla E\Sigma_{m+n}\ltimes_{\Sigma_{m+n-1}} Y
\]
on orbit spectra. There are also transfer maps
\[
\xymatrix{
E\Sigma_{m+n}\ltimes_{\Sigma_{m+n}} Y \ar[rr]^(.45){\tau_{m,n}}
 && E\Sigma_{m+n}\ltimes_{\Sigma_{m}\times\Sigma_{n} } Y \\
E\Sigma_{m+n}\ltimes_{\Sigma_{m+n}} Y \ar[rr]^(.45){\tau_{m+n-1,1}}
 && E\Sigma_{m+n}\ltimes_{\Sigma_{m+n-1} } Y
}
\]
associated with these inclusions of subgroups. We will
use the double coset formula of \cite{LNM1213}*{\S IV.6}.
We are in the situation of~\cite{LNM1213}*{theorem~IV.6.3},
and our first task is to identify representatives for
the double cosets in
\[
\Sigma_m\times\Sigma_n\backslash\Sigma_{m+n}/\Sigma_{m+n-1}.
\]
An elementary exercise with cycle notation shows that the
following are true:
\begin{itemize}
\item
the elements of $\Sigma_{m+n}/\Sigma_{m+n-1}$ are the distinct
left cosets $(r,m+n)\Sigma_{m+n-1}$ where $1\leq r\leq m+n-1$,
together with $\Sigma_{m+n-1}$;
\item
by definition, the elements of
$\Sigma_m\times\Sigma_n\backslash\Sigma_{m+n}/\Sigma_{m+n-1}$
are the $\Sigma_m\times\Sigma_n$-orbits in $\Sigma_{m+n}/\Sigma_{m+n-1}$
and these are represented by $(m,m+n)$ and $\id$. In fact the
orbit of $(m,m+n)$ contains all the $(r,m+n)$ with $1\leq r\leq m$,
and the orbit of the identity $I$ contains all of the transpositions
$(m+r,m+n)$ with $1\leq r\leq n$.
\end{itemize}

It is straightforward to verify the two identities
\begin{align*}
\Sigma_m\times\Sigma_{n-1} &=
 \Sigma_m\times\Sigma_n \cap \Sigma_{m+n-1}, \\
\Sigma_{m-1}\times\Sigma_n &=
 \Sigma_m\times\Sigma_n \cap (m,m+n)\Sigma_{m+n-1}(m,m+n).
\end{align*}
Now the double coset formula tells us that in the homotopy
category, there is a commutative diagram having the following
form.
\begin{equation}\label{eq:DoubCoset}
\xymatrix@C=0.3cm@R=1.0cm{
& E\Sigma_{m+n}\ltimes_{\Sigma_m\times\Sigma_n} Y\ar[drr]\ar[dl]_{\tau_{n-1,1}\vee\tau_{m-1,1}\;}
     && \\
E\Sigma_{m+n}\ltimes_{\Sigma_{m}\times\Sigma_{n-1}} Y \vee E\Sigma_{m+n}\ltimes_{\Sigma_{m-1}\times\Sigma_n} Y
\ar[d] & && E\Sigma_{m+n}\ltimes_{\Sigma_{m+n}} Y \ar[d]^{\tau_{m+n-1,1}} \\
E\Sigma_{m+n}\ltimes_{\Sigma_{m+n-1}} Y\vee E\Sigma_{m+n}\ltimes_{\Sigma_{m+n-1}} Y
\ar[rrr]^(.54){\mathrm{fold}} &&
& E\Sigma_{m+n}\ltimes_{\Sigma_{m+n-1}} Y
}
\end{equation}

\section{Power operations and the free functor}\label{sec:PowOps+PX}

We will describe another result on the effect of certain
transfer maps in homology that sheds light on the calculation
of universal derivations. We begin by recalling some standard
facts about the homology of extended powers.

Let $p$ be a prime and let $V=V_*$ be a graded $\F_p$-vector
space. The inclusion $C_p\leq\Sigma_p$ of the subgroup of
cyclic permutations $C_p=\langle\gamma\rangle$ with
$\gamma=(1,2,\ldots,p)$ has index $(p-1)!$, so the associated
transfer homomorphism provides a splitting for the induced
homomorphism in group homology with coefficients in the
$p$-fold tensor power $V^{\otimes p}$ with the obvious action.
\[
\xymatrix{
\mathrm{H}_*(C_p;V^{\otimes p})\ar@{->>}[r]
 & \ar@/_19pt/@{ >->}[l]_{\Tr_{C_p}^{\Sigma_p}}
                   \mathrm{H}_*(\Sigma_p;V^{\otimes p})
}
\]
Furthermore, the homology of the subgroup $\Sigma_{p-1}\leq\Sigma_p$
is trivial in positive degrees, \ie,
\[
\mathrm{H}_*(\Sigma_{p-1};V^{\otimes p})
     = \mathrm{H}_0(\Sigma_{p-1};V^{\otimes p})
     = (V^{\otimes p})_{\Sigma_{p-1}}.
\]
Hence the associated transfer homomorphism is also
zero in positive degrees, \ie, for $k>0$,
\begin{equation}\label{eq:Tr=0}
0=\Tr_{\Sigma_{p-1}}^{\Sigma_p}\:\mathrm{H}_k(\Sigma_p;V^{\otimes p})
     \to\mathrm{H}_k(\Sigma_{p-1};V^{\otimes p}).
\end{equation}
In fact the diagram of subgroup inclusions
\[
\xymatrix{
1\ar@{^(->}[r]\ar@{^(->}[d] & C_p\ar@{^(->}[d] \\
\Sigma_{p-1}\ar@{^(->}[r] & \Sigma_p
}
\]
induces a commutative diagram of split epimorphisms.
\[
\xymatrix{
V^{\otimes p}\ar[r]\ar@{->>}[d]
  & \mathrm{H}_*(C_p;V^{\otimes p})\ar@{->>}[d] \\
{\ph{\Sigma_{p-1}}}(V^{\otimes p})_{\Sigma_{p-1}}\ar[r]\ar@/^17pt/@{ >->}[u]^{\Tr_{1}^{\Sigma_{p-1}}}
  & \mathrm{H}_*(\Sigma_p;V^{\otimes p})\ar@/_17pt/@{ >->}[u]_{\Tr_{C_p}^{\Sigma_p}}
}
\]

This can be generalised to $\Sigma_{p^m}$ where $m\geq2$.
Then the $p$-order of $|\Sigma_{p^m}|$ is
\[
\ord_p|\Sigma_{p^m}| = \frac{(p^m-1)}{(p-1)}
                     = p^{m-1}+p^{m-2}+\cdots+p+1.
\]
Writing
\[
\Sigma_{p^{m-1}}^k =
\overset{k}{\overbrace{\Sigma_{p^{m-1}}\times\cdots\times\Sigma_{p^{m-1}}}},
\]
the wreath product
\[
\Sigma_p\wr\Sigma_{p^{m-1}} =
\Sigma_p\ltimes\Sigma_{p^{m-1}}^p \leq \Sigma_{p^m}
\]
has $p$-order
\[
\ord_p|\Sigma_p\wr\Sigma_{p^{m-1}}|
             = 1 + p\frac{(p^{m-1}-1)}{(p-1)}
             = \frac{(p^m-1)}{(p-1)},
\]
so an argument using transfer shows that the inclusion
induces a split epimorphism.
\[
\xymatrix{
H_*(\Sigma_p\wr\Sigma_{p^{m-1}};\F_p)\ar@{->>}[rr] &&
\ar@/_19pt/@{ >->}[ll]_{\Tr_{\Sigma_p\wr\Sigma_{p^{m-1}}}^{\Sigma_{p^m}}}
H_*(\Sigma_{p^m};\F_p)
}
\]
Another calculation shows that
\[
\ord_p|\Sigma_{p^m-1}| = \frac{(p^m-1)}{(p-1)} - m
                       = (p^{m-1}+p^{m-2}+\cdots+p+1) - m
\]
and
\begin{align*}
\ord_p|\Sigma_{p^{m-1}}^{(p-1)}\times\Sigma_{p^{m-1}-1}|
&= (p-1)\frac{(p^{m-1}-1)}{(p-1)} + \frac{(p^{m-1}-1)}{(p-1)} - (m-1) \\
&= (p^{m-1}+p^{m-2}+\cdots+p+1) - m \\
&= \ord_p|\Sigma_{p^m-1}|.
\end{align*}
Therefore
\[
\Sigma_{p^{m-1}}^{(p-1)}\times\Sigma_{p^{m-1}-1}\leq\Sigma_{p^m-1}
\]
and these subgroups of $\Sigma_{p^m}$ have the same $p$-order,
hence the inclusion induces an isomorphism
\[
H_*(\Sigma_{p^{m-1}}^{(p-1)}\times\Sigma_{p^{m-1}-1};\F_p)
            \xrightarrow{\;\iso\;} H_*(\Sigma_{p^m-1};\F_p).
\]
Consider the commuting diagram of subgroup inclusions
\[
\xymatrix@C=0.05cm@R=1.0cm{
 & \Sigma_{p^m} & \\
\Sigma_{p}\wr\Sigma_{p^{m-1}}\ar@{^{(}->}[ur]^{1}  && \\
  && \Sigma_{p^m-1}\ar@{_{(}->}[uul]_{p^m} \\
\Sigma_{p-1}\wr\Sigma_{p^{m-1}}\times\Sigma_{p^{m-1}}\ar@{^{(}->}[uu]^{p}
  && \\
&\Sigma_{p^{m-1}}^{(p-1)}\times\Sigma_{p^{m-1}-1}
 \ar@{_{(}->}[ul]_{p^{m-1}}\ar@{^{(}->}[uur]^{1}\ar@{^{(}->}[uuuu]^{p^{m}}
      & \\
}
\]
in which the arrows are decorated with the $p$-power factors
of the indices, \ie, if $H\leq G$ then the number would be
$p^{\ord_p|G:H|}$. Applying homology with coefficients in
$V^{\otimes p^m}$ with the evident action of $\Sigma_{p^m}$,
$\mathrm{H}_*(-;V^{\otimes p^m})$, we obtain a commutative
diagram of induced homomorphisms (solid arrows) and transfer
homomorphisms (dashed arrows).
%%%\begin{turn}{90}
\[
\xymatrix@C=0.05cm@R=1.0cm{
 & \mathrm{H}_*(\Sigma_{p^m};V^{\otimes p^m}) & \\
\mathrm{H}_*(\Sigma_{p}\wr\Sigma_{p^{m-1}};V^{\otimes p^m})
     \ar@<0.5ex>@{->>}[ur]\ar@<-0.5ex>@{<--<}[ur]
 && \\
 && \mathrm{H}_*(\Sigma_{p^m-1};V^{\otimes p^m})
 \ar@<-0.5ex>[uul]\ar@<0.5ex>@{<--}[uul]^{\Tr_{\Sigma_{p^m-1}}^{\Sigma_{p^m}}}\ar@<-0.5ex>@{ >-->}[ddl] \\
\mathrm{H}_*(\Sigma_{p-1}\wr\Sigma_{p^{m-1}}\times\Sigma_{p^{m-1}};V^{\otimes p^m})
\ar@<0.5ex>[uu]
\ar@<-0.5ex>@{<--}[uu]_{\Tr_{\Sigma_{p-1}\wr\Sigma_{p^{m-1}}\times\Sigma_{p^{m-1}}}^{\Sigma_{p}\wr\Sigma_{p^{m-1}}}}
 && \\
&\mathrm{H}_*(\Sigma_{p^{m-1}}^{(p-1)}\times\Sigma_{p^{m-1}-1};V^{\otimes p^m})
\ar@<0.5ex>[ul]\ar@<-0.5ex>@{<--}[ul]\ar@<-0.5ex>@{>>}[uur]
\ar@<-0.5ex>[uuuu] \ar@<0.5ex>@{<--}[uuuu] & \\
}
\]
%%%\end{turn}
As the transfer is contravariantly functorial with respect
to homomorphisms induced from inclusions, it is enough to
show that
$\Tr_{\Sigma_{p-1}\wr\Sigma_{p^{m-1}}\times\Sigma_{p^{m-1}}}^{\Sigma_{p}\wr\Sigma_{p^{m-1}}}$
is zero in positive degrees to deduce that the same holds
for $\Tr_{\Sigma_{p^m-1}}^{\Sigma_{p^m}}$. But this follows
since
\[
\mathrm{H}_*(\Sigma_{p}\wr\Sigma_{p^{m-1}};V^{\otimes p^m})
\iso
\mathrm{H}_*(\Sigma_{p};\mathrm{H}_*(\Sigma_{p^{m-1}};V^{\otimes p^{m-1}})^{\otimes p})
\]
and by~\eqref{eq:Tr=0} we already know the result for
all transfer homomorphisms of the form
\[
\Tr_{\Sigma_{p-1}}^{\Sigma_{p}}\:
\mathrm{H}_*(\Sigma_{p};W^{\otimes p})
 \to \mathrm{H}_*(\Sigma_{p-1};W^{\otimes p})
\]
for some $\F_p$-vector space~$W$. To summarise, we
have verified
\begin{lem}\label{lem:Tr=0}
For $m\geq1$, the transfer
\[
\Tr_{\Sigma_{p^m-1}}^{\Sigma_{p^m}}\:
\mathrm{H}_*(\Sigma_{p^m};V^{\otimes p^m})
\to
\mathrm{H}_*(\Sigma_{p^m-1};V^{\otimes p^m})
\]
is zero in positive degrees.
\end{lem}

Recall the standard $2$-periodic projective $\F_p[C_p]$-resolution
of~$\F_p$,
\begin{equation}\label{eq:Cp-resolution}
\xymatrix{
\F_p & \ar@{->>}[l] \F_p[C_p]e_0 &&\ar[ll]_{1-\gamma}\F_p[C_p]e_1
&&\ar[ll]_{1+\gamma+\cdots+\gamma^{p-1}}
\F_p[C_p]e_2 && \ar[ll]_(.4){1-\gamma} \cdots
}
\end{equation}
where $C_p=\langle\gamma\rangle$ generated by the $p$-cycle
$\gamma=(1,2,\ldots,p)$. Tensoring over $\F_p[C_p]$ gives
a complex
\[
\xymatrix{
0 & \ar@{->}[l] \F_pe_0\otimes V^{\otimes p}
&&\ar[ll]_{1-\gamma}\F_pe_1\otimes V^{\otimes p}
&&\ar[ll]_{1+\gamma+\cdots+\gamma^{p-1}}\F_pe_2\otimes V^{\otimes p}
&&\ar[ll]_(.4){1-\gamma}\cdots
}
\]
whose homology is $\mathrm{H}_*(C_p;V^{\otimes p})$.
%In each non-negative degree, the transfer is induced
%by the $\F_p$-linear mapping
%\[
%\F_pe_n\otimes V^{\otimes p}\to\F_p[C_p]e_n\otimes V^{\otimes p};
%\]
%given on basic tensors by
%\[
%e_n\otimes(v_1\otimes\cdots\otimes v_p) \mapsto
%(1+\gamma+\cdots+\gamma^{p-1})\otimes(v_1\otimes\cdots\otimes v_p).
%\]
%Here we regard \eqref{eq:Cp-resolution} as defining
%a projective $\F_p[1]=\F_p$-resolution and therefore
%\[
%\xymatrix{
%0 & \ar@{->>}[l] \F_p[C_p]e_0\otimes V^{\otimes p}
%&&\ar[ll]_{1-\gamma}\F_p[C_p]e_1\otimes V^{\otimes p}
%&&\ar[ll]_{1+\gamma+\cdots+\gamma^{p-1}}\F_p[C_p]e_2\otimes V^{\otimes p}
%&\ar[l]_(.4){1-\gamma} \cdots
%}
%\]
%has homology equal to
%\[
%\mathrm{H}_*(1;V^{\otimes p}) = \mathrm{H}_0(1;V^{\otimes p})
%                              = V^{\otimes p}.
%\]

Let $X$ be a connective cofibrant $S$-module, and let
$x\in H_n(X;\F_p)$ be a non-zero element. Recall the
algebraic results of~\cite{JPM:Steenrodops}*{lemma~1.4}.
If~$p$ is a prime, then there are elements
\[
e_r\otimes x^{\otimes p} =
e_r\otimes\overset{p}{\overbrace{x\otimes\cdots\otimes x}}
\]
which survive to non-zero homology classes in
$\mathrm{H}_*(C_p;H_*(X;\F_p)^{\otimes p})$ for $r\geq0$,
and where if~$p$ is odd,
\begin{itemize}
\item
$n$ is even and $r=2s(p-1)$ or $r=2(s+1)(p-1)-1$
for $0\leq s\in\Z$,
\item
$n$ is odd, $r=(2s+1)(p-1)$ or $r=(2s+1)(p-1)-1$
for $0\leq s\in\Z$.
\end{itemize}
These map to non-zero elements
\[
\tilde{\dlQ}_rx,\beta\tilde{\dlQ}_rx
       \in\mathrm{H}_*(\Sigma_p;H_*(X;\F_p)^{\otimes p})
\]
depending on the parity of $r$. There is a canonical
isomorphism
\[
\mathrm{H}_*(\Sigma_p;H_*(X;\F_p)^{\otimes p})
          \xrightarrow{\;\iso\;}
                  H_*(E\Sigma_p\ltimes_{\Sigma_p}X^{(p)})
\]
and we also denote the images of $\tilde{\dlQ}_sx,\tilde{\beta\dlQ}_sx$
by the same symbols. The natural weak equivalence
\[
E\Sigma_p\ltimes_{\Sigma_p}X^{(p)}
   \xrightarrow{\;\sim\;} X^{(p)}/\Sigma_p
\]
induces an isomorphism
\[
\xymatrix{
\mathrm{H}_*(\Sigma_p;H_*(X;\F_p)^{\otimes p})\ar[r]_{\iso}\ar@/^19pt/[rr]^{\iso}
& H_*(E\Sigma_p\ltimes_{\Sigma_p}X^{(p)};\F_p)\ar[r]_(.57){\iso}
& H_*(X^{(p)}/\Sigma_p;\F_p)
}
\]
sending $\tilde{\dlQ}_rx$ to the element which we will denote
by $\bar{\dlQ}_rx\in H_*(X^{(p)}/\Sigma_p;\F_p)$. When~$p$
is odd, whenever $2r\geq n$ we will set
\begin{align*}
\bar{\dlQ}^rx &= (-1)^r\nu(n)\bar{\dlQ}_{(2r-n)(p-1)}x,
 & \beta\bar{\dlQ}^rx &= (-1)^r\nu(n)\bar{\dlQ}_{(2r-n)(p-1)-1}x,
\end{align*}
in keeping with upper indexing for Dyer-Lashof operations,
where
\[
\nu(n) = (-1)^{n(n-1)(p-1)/4}\biggl(((p-1)/2)!\biggr)^n,
\]
which does not depend on~$r$. If $p=2$, whenever $r\geq n$
we set
\[
\bar{\dlQ}^rx = \bar{\dlQ}_{r-n}x.
\]
The action of the Dyer-Lashof operations $\dlQ^r$ and
$\beta\dlQ^r$ on $H_*(\mathbb{P}X;\F_p)$ described by
Steinberger in~\cite{LNM1176}*{chapter~III} is consistent
with this notation; however, we will sometimes write
$\dlQ^r\.x$ or $\beta\dlQ^r\.x$ when applying such an
operation to an element
$x\in H_*(X;\F_p)\subseteq H_*(\mathbb{P}X;\F_p)$ to
avoid potential confusion when $X$ is itself a
commutative $S$-algebra.
\begin{lem}\label{lem:DLops-compare}
For a connective cofibrant $S$-module $X$, and an element
$x\in H_*(X;\F_p)$, under the natural map
\[
\xymatrix{
E\Sigma_p\ltimes_{\Sigma_p}X^{(p)}
\ar[r]_(.57){\;\sim\;}\ar@/^19pt/[rr]^\rho
& X^{(p)}/\Sigma_p\ar[r]
& \mathbb{P}X
}
\]
in $H_*(-;\F_p)$ we have
\[
\dlQ^rx = \rho_*(\bar{\dlQ}^rx),
\quad
\beta\dlQ^rx = \rho_*(\beta\bar{\dlQ}^rx).
\]
\end{lem}
\begin{proof}
The basic observation is that a commutative $S$-algebra
is an algebra over the monad $\mathbb{P}\circ\mathbb{U}$
where the two model categories $\mathscr{M}_S$ and
$\mathscr{C}_S$ are related by the Quillen adjunction
of~\eqref{eq:QuillenAdj-C-M} with $R=S$.
\[
\xymatrix{
{\mathscr{C}_S} \ar@/_8pt/[rr]_{\mathbb{U}}
      && {\mathscr{M}_S} \ar@/_8pt/[ll]_{\mathbb{P}}
}
\]
The definition of the Dyer-Lashof operations for a
commutative $S$-algebra~$A$ involves the composition

\smallskip
\[
\xymatrix{
E\Sigma_p\ltimes_{\Sigma_p} (A^c)^{(p)}
     \ar[r]\ar@/^21pt/[rrrr]
& (A^c)^{(p)}/\Sigma_p \ar[r]
& A^{(p)}/\Sigma_p\ar[r]
& \mathbb{P}A\ar[r] & A
}
\]
where $(-)^c$ denotes cofibrant replacement in $\mathscr{M}_S$.
When $X$ is cofibrant in $\mathscr{M}_S$ and $A=\mathbb{P}X$,
we obtain a commutative diagram of the form
\[
\xymatrix{
E\Sigma_p\ltimes_{\Sigma_p} X^{(p)}\ar[d]\ar@/^21pt/[drrrr]^\rho
                                       &&& \\
E\Sigma_p\ltimes_{\Sigma_p} ((\mathbb{P}X)^c)^{(p)}\ar[r]
& ((\mathbb{P}X)^c)^{(p)}/\Sigma_p \ar[r]
& (\mathbb{P}X)^{(p)}/\Sigma_p\ar[r]
& \mathbb{P}(\mathbb{P}X)\ar[r] & \mathbb{P}X
}
\]
and applying $H_*(-;\F_p)$ to this gives the result.
\end{proof}

Using iterated extended powers associated with the iterated
wreath power subgroups
\[
\Sigma_p^{\wr p} =
\overset{\ell}{\overbrace{\Sigma_p\wr\cdots\wr\Sigma_p}}
                                \leq\Sigma_{p^\ell}
\]
we can form elements
\[
\dlQ^Ix =
\beta^{\epsilon_1}\bar{\dlQ}^{i_1}
   \cdots\beta^{\epsilon_\ell}\bar{\dlQ}^{i_\ell}x
     \in H_*(\mathbb{P}X;\F_p)
\]
for $x\in H_*(X;\F_p)$, where
$I = (\epsilon_1,i_1,\ldots,\epsilon_\ell,i_\ell)$,
$i_k>0$ and $\epsilon_k=0,1$ (with $\epsilon_k=0$
when $p=2$) can be interpreted in terms of the
homology of iterated wreath powers, and we obtain
the compatibility formula
\[
\dlQ^Ix = \rho^\ell_*(\bar{\dlQ}^Ix),
\]
where $\rho^\ell$ is defined in an obvious way from
the following diagram.
\[
\xymatrix{
E\Sigma_p^{\wr p}
  \ltimes_{\Sigma_p^{\wr p}} X^{(p^\ell)}
     \ar[d]\ar@/^21pt/[ddrrrr]^{\rho^\ell}
    &&& \\
E\Sigma_{p^\ell}\ltimes_{\Sigma_{p^\ell}} X^{(p^\ell)} \ar[d]
    &&& \\
E\Sigma_{p^\ell}\ltimes_{\Sigma_{p^\ell}}((\mathbb{P}X)^c)^{({p^\ell})}
                   \ar[r]
 & ((\mathbb{P}X)^c)^{({p^\ell})}/\Sigma_{p^\ell}\ar[r]
 & (\mathbb{P}X)^{({p^\ell})}/\Sigma_{p^\ell}\ar[r]
 & \mathbb{P}(\mathbb{P}X)\ar[r] & \mathbb{P}X
}
\]
We can now state an important result.
\begin{thm}\label{thm:deltaQ^Ix=0}
Let $X$ be a connective cofibrant $S$-module. Then
\[
(\delta_{(\mathbb{P}X,S)})_*\:H_*(\mathbb{P}X;\F_p)
                \to H_*(\Omega_S(\mathbb{P}X);\F_p)
\]
annihilates every element of the form $\dlQ^Ix$,
where $\len(I)>0$ and $x\in H_*(X;\F_p)$.
\end{thm}
\begin{proof}
This follows from the observation~\eqref{eq:delta-cpt}
together with Lemma~\ref{lem:Tr=0}.
\end{proof}

We will give an explicit description of the algebra
in Theorem~\ref{thm:H*PX-Fp}.

Recall that if $A$ is a connective commutative $S$-algebra
for which $\pi_0(A)$ is augmented over $\F_p$, there is
an induced augmentation of commutative $S$-algebras
$A\to H\F_p$, making $H\F_p$ into an $A$-module.
Theorem~\ref{thm:deltaQ^Ix=0} generalises to give the
following result.
\begin{thm}\label{thm:deltaQ^Ix=0-gen}
Let $A$ be a connective commutative $S$-algebra. Then
\[
(\delta_{(A,S)})_*\:H_*(A;\F_p)\to H_*(\Omega_S(A);\F_p)
\]
annihilates every element of the form $\bar{\dlQ}^Ia$ with
$\len(I)>0$ and $a\in H_*(A;\F_p)$. Hence, if
$H\F_p$ is an $A$-module the homomorphism
\[
\Delta_{(A,S)}\:H_*(A;\F_p)
       \to H^A_*(\Omega_S(A);\F_p) = \TAQ_*(A,S;H\F_p)
\]
also annihilates all such elements.
\end{thm}
\begin{proof}
Using the observation in the Proof of
Lemma~\ref{lem:DLops-compare}, we know there is a morphism
of commutative $S$-algebras $\mathbb{P}A\to A$ extending
the multiplication. Choose a cofibrant replacement $A^c\to A$
for the underlying $S$-module of $A$. By naturality there
is a commutative diagram in the homotopy category of $S$-modules
\[
\xymatrix{
\mathbb{P}A^c\ar[d]_{\delta_{(\mathbb{P}A^c,S)}}\ar[rr]
&& \mathbb{P}A\ar[d]_{\delta_{(\mathbb{P}A,S)}}\ar[rr]
&&  A\ar[d]_{\delta_{(A,S)}} \\
\mathbb{P}A^c\wedge A^c\ar[rr]
&& \Omega_S(\mathbb{P}A)\ar[rr]
&& \Omega_S(A)
}
\]
and on applying $H_*(-)=H_*(-;\F_p)$ we obtain an algebraic
commutative diagram.
\[
\xymatrix{
H_*(\mathbb{P}A^c)\ar[d]_{(\delta_{(\mathbb{P}A^c,S)})_*}\ar[rr]
&& H_*(\mathbb{P}A)\ar[d]_{(\delta_{(\mathbb{P}A,S)})_*}\ar[rr]
&& H_*(A)\ar[d]_{(\delta_{(A,S)})_*} \\
H_*(\mathbb{P}A^c\wedge A^c)\ar[rr]
&& H_*(\Omega_S(\mathbb{P}A))\ar[rr]
&& H_*(\Omega_S(A))
}
\]
Since an element of the form $\dlQ^Ia$ lifts back to an
element $\dlQ^Ia'\in H_*(\mathbb{P}A^c)$ as explained
above, the result follows. The result about $\theta'$ is
immediate from the definition~\eqref{eq:TAQ-Hurewiczhom}.
\end{proof}

\section{The reduced free commutative algebra functor}
\label{sec:redPX}

In this section we describe a modification $\tilde{\mathbb{P}}_R$
of the usual free functor $\mathbb{P}_R$ from $R$-modules to
commutative $R$-algebras. Recall that $\mathbb{P}_R$ is
left adjoint to the forgetful functor and when viewed as
an endofunctor of $\mathscr{M}_R$, its algebras are precisely
the commutative $R$-algebras. Then $\tilde{\mathbb{P}}_R$
gives an endofunctor on the comma category $S^0_R/\mathscr{M}_R$
of $R$-modules under the $R$-sphere $S^0_R$. Roughly speaking,
the difference between $\mathbb{P}_RX$ and $\tilde{\mathbb{P}}_RX$
is that in the latter, the morphism $S^0_R\to X$ becomes identified
with the unit. There are two reasons why this is useful to us.
First we will investigate the free algebra for an $R$-module
with such a `unit' morphism corresponding to a bottom cell (at
least when $R=S$). Second, a unital commutative $R$-algebra has
a natural choice of such morphism induced by the unit, and algebras
over $\tilde{\mathbb{P}}_R$ in $S^0_R/\mathscr{M}_R$ are the
commutative $R$-algebras; a similar observation was made by
\Kriz{} \& May~\cite{IK&JPM:Asterisque}, see the discussion
leading up to and including proposition~3.7.

Throughout, we fix a cofibrant commutative $S$-algebra~$R$.
The two model categories $\mathscr{M}_R$ and $\mathscr{C}_R$
are related by the Quillen adjunction
\[
\xymatrix{
{\mathscr{C}_R} \ar@/_8pt/[rr]_{\mathbb{U}}
         && {\mathscr{M}_R} \ar@/_8pt/[ll]_{\mathbb{P}_R}
}
\]
where the right adjoint $\mathbb{U}$ is the forgetful functor.
For a cofibrant $R$-module~$Z$, inclusion of the basepoint
$*\to X$ induces a cofibration of commutative $R$-algebras
$R=\mathbb{P}_R*\to\mathbb{P}_RZ$, so $\mathbb{P}_RZ$ is
cofibrant in the model category $\mathscr{C}_R$. More
generally a(n acyclic) cofibration $f\:X\to Y$ in
$\mathscr{M}_R$ induces a(n acyclic) cofibration
$\mathbb{P}_Rf\:\mathbb{P}_R X\to\mathbb{P}_RY$ in
$\mathscr{C}_R$.

In $\mathscr{M}_R$, $R$ is not cofibrant and we denote its
functorial cofibrant replacement by $S^0_R=\mathbb{F}_RS^0$
and a weak equivalence induced by a map of spectra $S_R^0\to R$
which represents the unit.
\[
\xymatrix{
{*}\ar@{ >->}[r]\ar[dr] & S^0_R\ar@{->>}[d]^{\sim}\\
   & R
}
\]
There is a unique induced morphism $\mathbb{P}_R S^0_R\to R$
in $\mathscr{C}_R$, but this need not be a cofibration.
Using the functorial factorisation in~$\mathscr{C}_R$
we obtain a commutative diagram in $\mathscr{C}_R$
\begin{equation}\label{eq:tildeR}
\xymatrix{
& \ar@{ >->}[dl]\mathbb{P}_R S^0_R\ar[dr] & \\
\widetilde{R}\ar@{->>}[rr]_{\sim} & & R
}
\end{equation}
which we use to fix the left hand arrow.

We will make use of the comma category $S^0_R/\mathscr{M}_R$
of $R$-modules under $S^0_R$, whose objects are the morphisms
$S^0_R\to X$ and whose morphisms are the commuting diagrams
\[
\xymatrix{
& \ar[dl] S^0_R\ar[dr] & \\
X\ar[rr] & & Y
}
\]
with initial object $\Id_{S^0_R}$ and terminal object $S^0_R\to*$.
This inherits a model structure from $\mathscr{M}_R$. Given
$i\:S^0_R\to X$ in $S^0_R/\mathscr{M}_R$, we obtain the induced
morphism $\mathbb{P}_Ri\:\mathbb{P}_RS^0_R\to\mathbb{P}_RX$ in
$\mathscr{C}_R$. If $i$ is a cofibration, then $i$ is cofibrant
in $S^0_R/\mathscr{M}_R$ and
$\mathbb{P}_Ri\:\mathbb{P}_RS^0_R\to\mathbb{P}_RX$ is a cofibration
in $\mathscr{C}_R$; we will then write $X/S^0_R$ for the cofibre
of~$i$.
%
%We will use the following notion which was used extensively
%in~\cite{AJB-JPM} for $p$-local spectra.
%\begin{defn}\label{defn:Hurweicz}
%A spectrum (or equivalently an $S$-module) $X$ is a \emph{Hurewicz
%spectrum} if $X$ is cofibrant, connective and $\pi_0(X)$ is cyclic
%and non-zero. We generalise this notion to $p$-local spectra (or
%equivalently $S_{(p)}$-module) in the obvious way.
%\end{defn}
%
%Let $X$ be a Hurewicz spectrum and let $i\:S^0\to X$
%be a map which represents a generator of $\pi_0(X)$.
We obtain a pushout diagram of commutative $R$-algebras
\[
\xymatrix{
\mathbb{P}_RS^0_R\ar@{}[dr]|{\PO}
%{\text{\Large$\ulcorner$}}
                  \ar[r]^{\mathbb{P}_Ri}\ar@{ >->}[d]
                        & \mathbb{P}_RX\ar@{ >->}[d] \\
\tilde{R}\ar[r] & \tilde{R}\wedge_{\mathbb{P}_RS^0_R}\mathbb{P}_RX
}
\]
and we set
\[
\tilde{\mathbb{P}}_RX =
   \tilde{R}\wedge_{\mathbb{P}_RS^0_R}\mathbb{P}_RX.
\]
If $i^c\:S^0_R\to X^c$ is the cofibrant replacement of~$i$
in the comma category, then the pushout diagram of solid
arrows in
\[
\xymatrix{
& & \mathbb{P}_RX\ar@/^19pt/@{ >-->}[ddd] \\
\mathbb{P}_RS^0_R\ar@{}[dr]|{\PO}
%{\text{\Large$\ulcorner$}}
\ar@/^19pt/@{-->}[urr]^{\mathbb{P}_Ri}\ar@{ >->}[r]^{\mathbb{P}_Ri^c}\ar@{ >->}[d]
         & \mathbb{P}_RX^c\ar@{ >->}[d]\ar@{-->}[ur] & \\
\tilde{R}\ar@{ >->}[r]\ar@/_19pt/@{-->}[drr]
& \tilde{R}\wedge_{\mathbb{P}_RS^0_R}\mathbb{P}_RX^c\ar@{-->}[dr] \\
& & \tilde{R}\wedge_{\mathbb{P}_RS^0_R}\mathbb{P}_RX
}
\]
defines the homotopy pushout of the first diagram,
\[
\tilde{\mathbb{P}}^h_RX =
   \tilde{R}\wedge_{\mathbb{P}_RS^0_R}\mathbb{P}_RX^c.
\]
Of course $\tilde{\mathbb{P}}^h_RX$ is well-defined
in the homotopy category $\bar{h}\mathscr{C}_R$, and
we have in effect defined it by making functorial
choices.

The model categories $S^0_R/\mathscr{M}_R$ and $\mathscr{C}_R$
are related by the Quillen adjunction
\[
\xymatrix{
{\mathscr{C}_R} \ar@/_8pt/[rr]_{\tilde{\mathbb{U}}}
 && {S^0_R/\mathscr{M}_R} \ar@/_8pt/[ll]_{\tilde{\mathbb{P}}_R}
}
\]
where the right adjoint $\tilde{\mathbb{U}}$ is the
forgetful functor sending $A$ to the composition

\[
\xymatrix{
S^0_R\ar[r]\ar@/^15pt/[rr] & R \ar[r] & A
}
\]
in $\mathscr{M}_R$. The total left derived functor of
$\tilde{\mathbb{P}}_R$ is $\tilde{\mathbb{P}}^h_R$
and
\[
\xymatrix{
{\bar{h}\mathscr{C}_R} \ar@/_8pt/[rr]_{\tilde{\mathbb{U}}^h}
 && {\bar{h}(S^0_R/\mathscr{M}_R)}\ar@/_8pt/[ll]_{\tilde{\mathbb{P}}^h_R}
}
\]
is a derived adjunction on homotopy categories.
%%%
%%%Using the canonical factorisation of $\mathbb{P}_RS^0_R\to R$
%%%into a cofibration and an acyclic fibration in $\mathscr{C}_R$,
%%%this can be replaced by a pushout diagram
%%%\[
%%%\xymatrix{
%%%& \ar@{.>}[dl]\mathbb{P}_RS^0_R
%%%\ar@{}[dr]|{\PO}
%%% %{\text{\Large$\ulcorner$}}
%%%\ar@{ >->}[r]^{\mathbb{P}_Ri}\ar@{ >->}[d]
%%%                             & \mathbb{P}_RX\ar@{ >->}[d] \\
%%%R & \ar@{->>}[l]_(.4)\sim R' \ar@{ >->}[r] & \tilde{\mathbb{P}}_RX
%%%}
%%%\]
%%%so that
%%%\[
%%%\tilde{\mathbb{P}}_RX = R' \wedge_{\mathbb{P}_RS^0_R}\mathbb{P}_RX
%%%\]
%%%is the homotopy pushout of the original diagram.

%An alternative homotopy equivalent model for $\tilde{\mathbb{P}}X$
%is given by the pushout
%\[
%\xymatrix{
%\mathbb{P}S^0
%\ar@{}[dr]|{\PO}
%%% %{\text{\Large$\ulcorner$}}
%%%\ar[r]^{i'}\ar@{ >->}[d]
%                             & \mathbb{P}X\ar@{ >->}[d] \\
%\mathbb{P}D^1\ar[r] & \tilde{\mathbb{P}}X
%}
%\]
%in which the vertical maps are cofibrations and the upper
%horizontal map $i'$ is induced from the composition
%\[
%S^0 \xrightarrow{\;\pinch\;} S^0\vee S^0
%   \xrightarrow{\;-\id\vee i\;}S \vee X \to \mathbb{P}X.
%\]
\begin{rem}\label{rem:S^0/pushouts}
In $S_R^0/\mathscr{M}_R$, pushouts are defined using pushouts
in $\mathscr{M}_R$, and we use the symbol $\undervee{S^0_R}$
to indicate such a pushout.
\[
\xymatrix{
S^0_R \ar[r] \ar[d]
\ar@{}[dr]|{\PO} & Y\ar[d] \\
X\ar[r] & X\undervee{S^0_R}Y
}
\]
Since the reduced free algebra functor
$\tilde{\mathbb{P}}_R\:\mathscr{M}_R\to\mathscr{C}_R$ is
a left adjoint it preserves pushouts, so for two $R$-modules
$X,Y$ under $S^0_R$,
\begin{equation}\label{eq:S^0/pushouts}
\tilde{\mathbb{P}}_R(X\undervee{S^0_R}Y) \iso
      \tilde{\mathbb{P}}_RX \wedge_R \tilde{\mathbb{P}}_RY.
\end{equation}
In particular, if we have a connective CW $R$-module with
distinguished bottom cell $S^0_R\to X$, then the $n$-skeleton
$X^{[n]}$ gives rise to the $n$-skeleton of the CW commutative
$R$-algebra $\tilde{\mathbb{P}}_R X$,
\begin{equation}\label{eq:S^0/skeleta}
\tilde{\mathbb{P}}_R^{\<n\>} X
    = (\tilde{\mathbb{P}}_R X)^{\<n\>}
    = \tilde{\mathbb{P}}_R(X^{[n]}).
\qedhere
\end{equation}
\end{rem}

We already know that for cofibrant $X$, in the homotopy
category $\bar{h}\mathscr{M}_{\mathbb{P}_RX}$,
\[
\Omega_R(\mathbb{P}_RX) \iso \mathbb{P}_RX\wedge_R X.
\]
\begin{prop}\label{prop:Omega-redP}
Let $X\in S^0_R/\mathscr{M}_R$ be cofibrant. Then in
the homotopy category $\bar{h}\mathscr{M}_{\tilde{\mathbb{P}}_RX}$,
\[
\Omega_R(\tilde{\mathbb{P}}_RX) \iso \tilde{\mathbb{P}}_R X\wedge_R X/S^0_R.
\]
\end{prop}
\begin{proof}
First we recall some observations appearing in~\cite{BGRtaq}.
For a pushout diagram of cofibrations of commutative
$R$-algebras
\[
\xymatrix{
A \ar@{ >->}[r] \ar@{ >->}[d]
\ar@{}[dr]|{\PO}
%{\text{\Large$\ulcorner$}}
                  & B\ar@{ >->}[d] \\
C\ar@{ >->}[r] & B\wedge_A C
}
\]
by \cite{MBtaq}*{proposition~4.6} we have
\begin{equation}\label{eq:Omega-redP-1}
\Omega_B(B\wedge_A C) \sim B\wedge_A \Omega_A(C).
\end{equation}

%%%When $A=\mathbb{P}_RZ$,
%%%for a cofibrant $R$-module~$Z$, and $C=\mathbb{P}_RCZ$
%%%(where $CZ$ denotes the cone on $Z$) with the natural
%%%inclusion $\mathbb{P}_RZ\to\mathbb{P}_RCZ$, this gives
%%%\begin{align}
%%%\Omega_{\mathbb{P}_RZ}(B\wedge_{\mathbb{P}_RZ}\mathbb{P}_RCZ)
%%%   &\sim B\wedge_{\mathbb{P}_RZ}\Omega_{\mathbb{P}_RZ}(\mathbb{P}_RCZ)
%%%       \notag \\
%%%   &\sim B\wedge_{\mathbb{P}_RZ}\mathbb{P}_RCZ \wedge_R \Sigma Z,
%%%%\label{eq:OmegatildePX}
%%%\end{align}
%%%where the last equivalence follows from~\cite{BGRtaq}*{(1.16)}.

%Now take $Z=S^0_R$, assume that $i\:S^0\to X$ is a cofibration,
%and also take $B=\mathbb{P}_RX$ so that
Now assume that $i\:S^0_R\to X$ is a cofibration, and consider
\[
\tilde{\mathbb{P}}_RX =
\tilde{R}\wedge_{\mathbb{P}_RS^0_R}\mathbb{P}_RX.
\]
Notice that by \eqref{eq:Omega-redP-1},
\[
\Omega_{\mathbb{P}_RX}(\tilde{\mathbb{P}}_RX) \sim
\mathbb{P}_RX\wedge_{\mathbb{P}_RS^0_R}\Omega_{\mathbb{P}_RS^0_R}(\tilde{R}),
\]
and since there is a weak equivalence of $R$-algebras
$\tilde{R}\xrightarrow{\sim}R$, the sequence
\[
R\to\mathbb{P}_RS^0_R\to\tilde{R}
\]
has an associated cofibre sequence of the form
\[
\xymatrix{
\tilde{R}\wedge_{\mathbb{P}_RS^0_R}\Omega_R(\mathbb{P}_RS^0_R)\ar[r] &
\Omega_R(\tilde{R})\ar[r] & \Omega_{\mathbb{P}_RS^0_R}(\tilde{R})\ar[r] &
\cdots
}
\]
where
\[
\Omega_R(\tilde{R})\sim\Omega_R(R)\sim *.
\]
Therefore
\[
\Omega_{\mathbb{P}_RS^0_R}(\tilde{R}) \sim
\tilde{R}\wedge_{\mathbb{P}_RS^0_R}\Sigma\Omega_R(\mathbb{P}_RS^0_R)
\sim \tilde{R}\wedge_{\mathbb{P}_RS^0_R}\mathbb{P}_RS^0_R\wedge_R\Sigma S^0_R
\sim \tilde{R}\wedge_R\Sigma S^0_R,
\]
and so
\[
\Omega_{\mathbb{P}_RX}(\tilde{\mathbb{P}}_RX) \sim
  \tilde{\mathbb{P}}_RX\wedge_R\Sigma S^0_R.
\]
Similarly, from the sequence
\[
R \to \mathbb{P}_RX \to \tilde{R}\wedge_{\mathbb{P}_RS^0_R}\mathbb{P}_RX
\]
we obtain a cofibre sequence of $\tilde{\mathbb{P}}_RX$-modules
\[
\xymatrix@C=0.7cm@R=1.0cm{
\tilde{\mathbb{P}}_RX \wedge_{\mathbb{P}_RX}\Omega_R(\mathbb{P}_RX)\ar[r]
& \Omega_R(\tilde{\mathbb{P}}_RX) \ar[r]
& \Omega_{\mathbb{P}_RX}(\tilde{\mathbb{P}}_RX)\ar[r]
& \tilde{\mathbb{P}}_RX \wedge_{\mathbb{P}_RX}\Sigma\Omega_R(\mathbb{P}_RX)
\ar[r] & \cdots
}
\]
which is equivalent to
\[
\xymatrix{
\tilde{\mathbb{P}}_RX \wedge_R X \ar[r]
 & \Omega_R(\tilde{\mathbb{P}}_RX) \ar[r]
 & \tilde{\mathbb{P}}_RX\wedge_R \Sigma S^0_R \ar[r]
 & \tilde{\mathbb{P}}_RX \wedge_R \Sigma X \ar[r]
 & \cdots
}
\]
and so
\[
\Omega_R(\tilde{\mathbb{P}}_RX) \sim
              \tilde{\mathbb{P}}_RX \wedge_R X/S^0_R,
\]
as required.
\end{proof}

Now let $A$ be a commutative $R$-algebra. On composing
the unit $R\to A$ with the weak equivalence $S^0_R\to R$
we obtain the object
\[
S^0_R\xrightarrow{\;\sim\;}R
                    \xrightarrow{\;\ph{\sim}\;} A
\]
in $S^0_R/\mathscr{M}_R$. Using functorial factorisation
we obtain a cofibrant replacement
\[
\xymatrix{
& \ar@{ >->}[dl] S^0_R\ar[dr] & \\
A^c\ar@{->>}[rr]^{\sim} && A
}
\]
and so
\[
\tilde{\mathbb{P}}^h_RA =
     \tilde{R}\wedge_{\mathbb{P}_RS^0_R}\mathbb{P}_RA^c.
\]
\begin{rem}\label{rem:tildePA->A}
The multiplication on a commutative $R$-algebra~$A$ extends
to a morphism of commutative $R$-algebras
$\tilde{\mathbb{P}}_RA\to A$. This follows from the evident
commutative diagram of solid arrows
\[
\xymatrix{
\mathbb{P}_RS^0_R
\ar[r]\ar@{ >->}[d]\ar@{}[dr]|{\PO}
%{\text{\Large$\ulcorner$}}
& \mathbb{P}_RA\ar[d]\ar@/^17pt/[ddr] &   \\
\tilde{R} \ar[r]\ar@{->>}[d]^{\sim} & \tilde{\mathbb{P}}_RA\ar@{-->}[dr] &  \\
R\ar@/_17pt/[rr] &&  A
}
\]

\bigskip
\noindent
where the curved arrows come from the unit and extension
of the product respectively.
\end{rem}

%%%%%%%%%%%%%%%%%%
%\section{The free commutative $S$-algebra functor and
%\protect$\Omega^\infty\Sigma^\infty Z$} \label{sec:PX&QZ}

For completeness we recall a useful relationship between
the free functor on a suspension spectrum and the suspension
spectrum of the associated infinite loop space. This can
be found in~\cite{NJK:LocAQG}*{section~4} for example.

Let $Z$ be a connected based space. The infinite loop space
$\Omega^\infty \Sigma^\infty Z$ gives rise to a commutative
$S$-algebra $\Omega^\infty\Sigma^\infty_+Z$, \ie, the
suspension spectrum of the based space
$\Omega^\infty\Sigma^\infty Z$ with a disjoint basepoint.

The natural (based) map $Z\to\Omega^\infty\Sigma^\infty Z$
viewed as an unbased map induces a based map
\[
\Sigma^\infty_+Z\to\Sigma^\infty_+\Omega^\infty\Sigma^\infty Z
\]
which extends uniquely to a morphism of ring spectra
\[
\mathbb{P}\Sigma^\infty_+Z\to\Sigma^\infty_+\Omega^\infty\Sigma^\infty Z.
\]
The base points in $Z$ and $\Omega^\infty\Sigma^\infty Z$
pick out maps from the sphere $S$ and there is a commutative
diagram of commutative $S$-algebras
\begin{equation}\label{eq:PZ-QZ}
\xymatrix{
\mathbb{P}S^0\ar[r]\ar@{ >->}[d]\ar@{}[dr]|{\PO}
%{\text{\Large$\ulcorner$}}
& \mathbb{P}\Sigma^\infty_+ Z\ar@/^17pt/[dddr]\ar@{-->}[d] & \\
\tilde{S}\ar@{->>}[d]_{\sim}\ar@/_17pt/[ddrr]\ar@{-->}[r]
& \tilde{\mathbb{P}}\Sigma^\infty_+ Z \ar@{-->}[ddr]&\\
S\ar@/_17pt/[drr]&& \\
&& \mathbb{P}\Sigma^\infty_+\Omega^\infty\Sigma^\infty Z
}
\end{equation}
where the acyclic fibration $\tilde{S}\to S$
is that defined in~\eqref{eq:tildeR}.

\begin{prop}\label{prop:PZ-QZ}
For a connected based space $Z$, the morphism
\[
\tilde{\mathbb{P}}\Sigma^\infty Z
    \xrightarrow{\;\sim\;}
     \Sigma^\infty_+\Omega^\infty\Sigma^\infty Z
\]
of~\eqref{eq:PZ-QZ} is a weak equivalence.
\end{prop}
\begin{proof}
If $\k=\Q$ and $\k=\F_p$ for $p$ a prime, comparison of the
known answers for $H_*(\tilde{\mathbb{P}}\Sigma^\infty Z;\k)$
and $H_*(\Omega^\infty\Sigma^\infty Z;\k)$ given in
Section~\ref{sec:HomologyPX} shows that this morphism induces
an isomorphism $H_*(-;\k)$. It follows that it induces an
isomorphism on $H_*(-;\Z)$, hence it is a weak equivalence.
\end{proof}

\section{The ordinary homology of free commutative \protect$S$-algebras}
\label{sec:HomologyPX}

In this section we recall some results on ordinary homology 
of free commutative $S$-algebras. 

For a commutative ring $\k$, and a graded $\k$-module~$V_*$,
we will write $\k\langle V_*\rangle$ for the free commutative
graded $\k$-algebra on $V_*$. If $V_*$ is connective and
$V_0$ is a cyclic $\k$-module, we set $\tilde{V}_* = V_*/V_0$.

Let $X$ be a cofibrant connective spectrum.
\begin{thm}\label{thm:H*PX-Q}
The rational homology of\/ $\mathbb{P}X$ is given by
\[
H_*(\mathbb{P}X;\Q) = \Q\langle H_*(X;\Q)\rangle.
\]
\end{thm}

In positive characteristic, the next result is fundamental.
We use the standard convention for Dyer-Lashof monomials
so that in an indexing sequence
\[
I =
(\epsilon_1,i_1,\epsilon_2,i_2,\ldots,\epsilon_\ell,i_\ell),
\]
each $i_r$ is positive, when $p=2$ all the $\epsilon_i$
are zero, while for odd~$p$, $\epsilon_i=0,1$. The length
of $\dlQ^I$ is $\len(\dlQ^I)=\ell$.
\begin{thm}\label{thm:H*PX-Fp}
For~$p$ a prime, $H_*(\mathbb{P}X;\F_p)$ is the free
commutative graded $\F_p$-algebra generated by elements
$\dlQ^Ix_j$, where $x_j$ for $j\in J$ gives a basis
for $H_*(X;\F_p)$ and
$I=(\epsilon_1,i_1,\epsilon_2,\ldots,\epsilon_\ell,i_\ell)$
is admissible with $\exc(I)+\epsilon_1>|x_j|$.
\end{thm}

So for $p=2$, this gives the polynomial ring
\[
H_*(\mathbb{P}X;\F_2) =
\F_2[\dlQ^Ix_j:\text{$j\in J$, $\exc(I)+\epsilon_1>|x_j|$}\;].
\]

Of course these results are very similar to those for the
homology of $\Omega^\infty\Sigma^\infty Z$ for a space~$Z$;
for a convenient overview of the latter, see~\cite{JPM:HomOps}.
The Dyer-Lashof operations are compatible with the notation
used in Section~\ref{sec:PowOps+PX}.

\begin{proof}[Sketch of why \emph{Theorems~\ref{thm:H*PX-Q}}
and~\emph{\ref{thm:H*PX-Fp}} hold]
We learnt some of the following from Mike Mandell, see
also~\cite{MBtaq}*{section~6}.

Let $R$ be a commutative $S$-algebra. The free functor
$\mathbb{P}_R$ for, sends $R$-modules to commutative
$R$-algebras. As it is a left adjoint it preserves
pushouts, so for $R$-modules $X,Y$,
\[
\mathbb{P}_R(X\vee Y) \iso
               \mathbb{P}_R(X)\wedge_R \mathbb{P}_R(Y).
\]
For any commutative $R$-algebra $A$, base change gives
\[
A\wedge_R \mathbb{P}_R(-) = \mathbb{P}_A(A\wedge_R(-)).
\]
If $R=S$ and $H=H\k$ for a field $\k$, then
\[
H\wedge\mathbb{P}(-) = \mathbb{P}_H(H\wedge (-)).
\]
Applying $\pi_*(-)$ gives a functor sending spectra
to commutative graded $\k$-algebras,
\[
X \mapsto H_*(\mathbb{P}(X)) = \pi_*(\mathbb{P}_H(H\wedge(-)),
\]
which preserves pushouts, in particular it sends wedges
to tensor products. For any spectrum $X$, as an $H$-module,
$H\wedge X$ is equivalent to a wedge of suspensions of $H$,
so the calculation of $H_*(\mathbb{P}(X))=H_*(\mathbb{P}(X);\k)$
reduces to that for spheres. For $\k=\Q$ this gives the
rational result.

When $\k=\F_p$, for a sphere $S^n$ the answer is the free
commutative $\k$-algebra on admissible Dyer-Lashof monomials
$\dlQ^I$ applied to an element~$s_n$, \ie, elements of the
form $\dlQ^Is_n=\dlQ^I\.s_n$, where the indexing
sequences~$I$ are admissible and satisfy $\exc(I)>|s_n|=n$.
Although this makes sense for $n\in\Z$ (for the general case
see~\cite{LNM1176}*{chapter~III}), we only require the case
where $n\geq0$.

Using the ideas of Section~\ref{sec:PowOps+PX}, we know
that for $x\in H_*(X)$, the element $\dlQ^Ix$ is the image
under the canonical homomorphism
\[
\xymatrix{
H_*(E\Sigma_{p^\ell}\ltimes_{\Sigma_{p^\ell}}X^{(p^\ell)})\ar[rr]^(.55){\iso}\ar[drr]
                && H_*(X^{(p^\ell)}/\Sigma_{p^\ell})\ar@{^{(}->}[d] \\
                && H_*(\mathbb{P}X)
}
\]
of an element obtained by forming iterated wreath powers
of~$x$ in the homology
\[
H_*(E\Sigma_p\ltimes_{\Sigma_p}(E\Sigma_{p^k}\ltimes_{\Sigma_{p^k}}X^{(p^k)})).
\]
Of course we can view $\bar{\dlQ}^Ix$ as obtained from~$x$
by applying the Dyer-Lashof monomial
\[
\dlQ^I =
\beta^{\epsilon_1}\dlQ^{i_1}\cdots\beta^{\epsilon_\ell}\dlQ^{i_\ell}
\]
which exists as an operation on the homology of any commutative
$S$-algebra.
\end{proof}

We will describe the analogous results for $\tilde{\mathbb{P}}X$
where $X\in S^0/\mathscr{M}_S$ is a cofibrant $S$-module under
$S^0$, i.e., equipped with a cofibration $S^0\to X$. Of course
our results can also be interpreted in the category of modules
over the $p$-local sphere.

Our main computational tool is the K\"unneth spectral sequence,
and we will use its multiplicative properties and compatibility
with the action of the Dyer-Lashof operations, see~\cite{MB-MM:BP-E4}
for some related results on this.

We begin by stating the rational result whose proof we leave
to the reader.
\begin{thm}\label{thm:H*redPX-Q}
The rational homology of\/ $\tilde{\mathbb{P}}X$ is given by
\[
H_*(\tilde{\mathbb{P}}X;\Q) = \Q\langle \tilde{H}_*(X;\Q)\rangle.
\]
\end{thm}

The positive characteristic case is of course more interesting.
\begin{thm}\label{thm:H*redPX-Fp}
Let~$p$ be a prime. If $X\in S^0/\mathscr{M}_S$ is cofibrant,
then $H_*(\tilde{\mathbb{P}}X;\F_p)$ is the free commutative
graded $\F_p$-algebra generated by elements $\dlQ^Ix_j$, where
$x_j$ for $j\in J$ gives a basis for
$\tilde{H}_*(X;\F_p)=H_*(X/S^0;\F_p)$ and
$I=(\epsilon_1,i_1,\epsilon_2,\ldots,\epsilon_\ell,i_\ell)$
is admissible with $\exc(I)+\epsilon_1>|x_j|$.
\end{thm}
\begin{proof}
We set $H_*(-) = H_*(-;\F_p)$.

Since the pushout agrees with the smash product over $\mathbb{P}S^0$,
there is a first quadrant K\"unneth spectral sequence with
\[
\mathrm{E}^2_{s,t} =
   \Tor^{H_*(\mathbb{P}S^0)}_{s,t}(H_*(\mathbb{P}X),\F_p)
    \Lra H_{s+t}(\tilde{\mathbb{P}}X).
\]
Here $i'_*\:H_*(\mathbb{P}S^0)\to H_*(\mathbb{P}X)$ embeds
the domain as a subalgebra since $x_0=i_*(s_0)$ generates
$H_0(X)=\F_p$ and
\begin{align*}
i'_*(s_0) &= x_0-1,   \\
i'_*(\dlQ^Is_0) &= \dlQ^I(x_0)-\dlQ^I(1) = \dlQ^I(x_0)
                   \quad \text{if $\len(I)>0$}.
\end{align*}
By the freeness of $H_*(\mathbb{P}X)$, it is a free
$i'_*H_*(\mathbb{P}S^0)$-module, hence
\begin{align*}
\mathrm{E}^2_{*,*} &= \Tor^{H_*(\mathbb{P}S^0)}_{0,*}(H_*(\mathbb{P}X),\F_p) \\
&= H_*(\mathbb{P}X)\otimes_{i'_*H_*(\mathbb{P}S^0)}\F_p \\
&= H_*(\mathbb{P}X)/(x_0-1,\dlQ^Ix_0:\len(I)>0).
\end{align*}
Thus the spectral sequence collapses at the $\mathrm{E}^2$-term
and the result follows.
\end{proof}

Armed with Theorem~\ref{thm:deltaQ^Ix=0}, we have
\begin{thm}\label{thm:univder-Q^I}
Let $p$ be a prime and let $X$ be a $p$-local connective
cofibrant $S$-module. Then the universal derivation
$\delta_{(\tilde{\mathbb{P}}X,S)}$ induces the derivation
\[
\Delta_{(\tilde{\mathbb{P}}X,S)}\: H_*(\tilde{\mathbb{P}}X;\F_p)
\to \TAQ_*(\tilde{\mathbb{P}}X,S;H\F_p) \iso H_*(X/S^0;\F_p)
\]
which acts on the elements $\dlQ^Ix$ with $x\in H_*(X;\F_p)$
by the rule
\[
\Delta(\dlQ^Ix) =
\begin{cases}
x & \text{\rm if $\len(I)=0$}, \\
0 & \text{\rm if $\len(I)>0$}.
\end{cases}
\]
\end{thm}

We remark that the universal derivation is not 
an $\mathcal{A}(p)_*$-comodule homomorphism. 
The interaction between Dyer-Lashof operations 
and the $\mathcal{A}(p)_*$-coaction is described 
in~\cite{Nishida}.

\section{Some calculations}\label{sec:Calculations}

Equipped with our earlier results, we revisit some 
of the calculations of~\cite{BGRtaq}*{section~5}.

First we consider the $\TAQ$-Hurewicz homomorphism 
for the \Einfty  Thom spectrum $MU$. By work of
Basterra and Mandell, in the homotopy category
$\bar{h}\mathscr{M}_{MU}$,
\[
\Omega_S(MU) \iso MU\wedge \Sigma^2 ku.
\]
At the prime~$p=2$, the Hurewicz homomorphism factors through
$H_*(MU;\F_2)$
\[
\xymatrix{
\pi_*(MU) \ar[r]\ar@/^19pt/[rrr]^\theta & H_*(MU;\F_2) \ar[r]_(.45){\theta'}
                         & \TAQ_*(MU,S;\H\F_2) \ar[r]_(.56){\iso}
                         & H_*(\Sigma^2 ku;\F_2)
}
\]
where
\[
\xymatrix{
H_*(MU;\F_2)\ar[r]^{\theta'}\ar@{=}[d] & H_{*-2}(\Sigma^2 ku;\F_2)\ar@{=}[d] \\
\F_2[b_r:r\geq1] & \Sigma^2\F_2[\zeta_1^2,\zeta_2^2,\zeta_3,\ldots]
}
\]
is a derivation and
\[
\theta'(b_r) =
\begin{cases}
\Sigma^2\xi_s^2 & \text{if $r = 2^s$}, \\
\;\;0 & \text{otherwise}.
\end{cases}
\]
Here $\zeta_s=\chi(\xi_s)$ is the conjugate of the Milnor generator
in $\F_2[\zeta_1^2,\zeta_2^2,\zeta_3,\ldots]\subseteq\mathcal{A}(2)_*$.
This tells us that the elements $b_{2^s}$ are not decomposable in
terms of the Dyer-Lashof action, and recovers part of Kochman's
result Theorem~\ref{thm:DL-MU-QH}.

For an odd prime $p$, the $\TAQ$-Hurewicz homomorphism decomposes
into $(p-1)$ pieces corresponding to the Adams splitting of
$p$-local connective $K$-theory, giving
\[
\Sigma^2ku_{(p)} \sim \bigvee_{1\leq r\leq p-1}\Sigma^{2r}\ell.
\]
This gives an equivalent homomorphism
\[
\theta'\:H_*(MU;\F_p) \to
     \bigoplus_{1\leq r\leq p-1}H_*(\Sigma^{2r}\ell;\F_p).
\]
Here
\[
H_*(\ell;\F_p) = \F_p[\zeta_i:i\geq1]\otimes\lambda_{\F_p}(\bar{\tau}_j:j\geq2)
   \subseteq\mathcal{A}(p)_*,
\]
and the component corresponding to $\Sigma^{2r}\ell$ is
determined in terms of generating functions by
\[
\sum_{i\geq0}b_it^i \mapsto t^r\biggl(\sum_{j\geq0}\xi_jt^{p^j-1}\biggr)^r.
\]
It follows that $b_k\mapsto0$ unless $k\equiv r\bmod{(p-1)}$.
Write $k=np^e$ with $p\nmid n$, so $n\equiv r\bmod{(p-1)}$.
Now set
\[
n-r = (p-1)(s_0+s_1p+\cdots+s_dp^d)
\]
with $0\leq s_i\leq p-1$ and $s_d\neq0$. Then we obtain
\[
b_{(s(p-1)+r)p^e} \mapsto
\binom{r}{r-s_0,s_0-s_1,\ldots,s_{d-1}-s_d,s_d}
\xi_e^{r-s_0}\xi_{e+1}^{s_0-s_1} \cdots
\xi_{e+d-1}^{s_{d-2}-s_{d-1}}\xi_{e+d}^{s_{d-1}-s_d}\xi_{e+d+1}^{s_d}.
\]
Notice that this can only give a non-zero answer if the
following inequalities are satisfies:
\[
1\leq s_d\leq s_{d-2}\leq\cdots \leq s_1\leq s_0\leq r.
\]
In these cases $b_{(s(p-1)+r)p^e}$ must be Dyer-Lashof
indecomposable, and so we again recover Kochman's odd
primary result of Theorem~\ref{thm:DL-MU-QH}.

Here is another example, the reader is invited to compare
it with that of $MU_{(2)}$ and $MSp_{(2)}$ in~\cite{BGRtaq}*{proposition~5.1}.
\begin{prop}\label{prop:MSU2}
The $2$-local commutative $S$-algebra $MSU_{(2)}$ is not
minimal atomic.
\end{prop}
\begin{proof}
We recall that $H_*(MSU;\F_2)$ is a polynomial algebra with
a generator in each even degree greater than~$2$. There are
many explicit generating families known, for example
see~\cites{AB:HomGenBSO+BSU,AB:HWdecomp}. In fact, $H_*(MSU;\F_2)$
can be identified as a subalgebra of $H_*(MU;\F_2)$, and then
there are polynomial generators $a_n\in H_{2n}(MU;\F_2)$ so
that
\[
H_*(MSU;\F_2) =
\F_2[a_{2^s}^2:s\geq0]\otimes\F_2[a_{2^sk}:\text{$s\geq0$, $k>1$ odd}]
                                         \subseteq H_*(MU;\F_2).
\]
We will write $a'_n$ the generator in degree $2n$ where
$n\geq2$, for our purposes it is not important which
choice we make here.

By \cite{SOK:DLops}*{theorem~19(a)}, the Dyer-Lashof indecomposables
in $H_*(MSU;\F_2)$ are the algebra generators appearing in degrees
of the form $2^m+2^n$ where $m,n\geq0$; this includes the case
$2^s=2^{s-1}+2^{s-1}$ where $s\geq1$.

Since there is a weak equivalence of infinite loop spaces
$BSU\sim\Omega^\infty\Sigma^4ku$, by~\cite{MB-MM:TAQ},
\[
\Omega_S(MSU) \iso MSU\wedge \Sigma^4ku.
\]
Therefore the $\TAQ$-Hurewicz homomorphism factors as
\[
\xymatrix{
\pi_*(MSU) \ar[r]\ar@/^19pt/[rrr]^{\theta}
       & H_*(MSU;\F_2)\ar[r]_(.45){\theta'}
       & \TAQ_*(MSU,S;H\F_2)\ar[r]_(.6){\iso}
       & H_*(\Sigma^4ku;\F_2)
}
\]
and in fact using the geometrically defined generators
described in~\cite{AB:HomGenBSO+BSU} it can be shown
that
\[
\im\theta' =\F_2\{\Sigma^4\xi_m^2\xi_n^2:m,n\geq1\}.
\]
Here $\theta'$ has the effect
\[
a'_{2^{n+1}} \mapsto \Sigma^4\xi_n^4\quad (n\geq0),
\qquad
a'_{2^m+2^n} \mapsto \Sigma^4\xi_m^2\xi_n^2\quad (n>m\geq0).
\]
However, this alone does not give us the result. We would
like to use~\cite{BGRtaq}*{theorems~3.2,3.4}, so we must show
that
\[
\theta\:\pi_n(MSU)\to\TAQ_n(MSU,S;H\F_2)
\]
is non-trivial for some $n>0$. For this we will use work of
Pengelley~\cite{DJP:MSU} on the Adams spectral sequence for
$\pi_n(MSU_{(2)})$. In~\cite{DJP:MSU}*{theorem~2.6} it is
shown that there are polynomial generators
$y'_{8k}\in H_{8k}(MSU;\F_2)$ for which the Adams differential
$d_2$ satisfies
\[
d_2y'_{8k} =
\begin{cases}
hq'_{s-1}\neq0 & \text{if $k=2^s$}, \\
0 & \text{if $k$ is not a power of $2$},
\end{cases}
\]
and furthermore all trivial higher differentials in the spectral
sequence are trivial. For our purposes what matters here is that
each generator $y'_{2^{m+3} + 2^{n+3}}$ where $n>m\geq0$ is in
the image of the classical Hurewicz homomorphism and under the
$\TAQ$-homomorphism it maps to $\Sigma^4\xi_{m+2}\xi_{n+2}\neq0$.
This means that $MSU_{(2)}$ cannot be minimal atomic.
\end{proof}

If the summand $BoP$ of $MSU_{(2)}$ were to have a commutative
$S$-algebra structure, then Pengelley's results would imply that
the $\bmod{\;2}$ $\TAQ$-Hurewicz homomorphism was trivial, hence
$BoP$ would be minimal atomic. However, this depends on the
observation that the classical $\bmod{\;2}$ Hurewicz homomorphism
is trivial so we already know it is minimal atomic as a
spectrum~\cite{AJB-JPM} and hence it would be as a commutative
$S$-algebra. So the use of $\TAQ$ would not be really necessary.

Here are some more examples.
\begin{examp}\label{examp:HFp}
Let $p$ be a prime and set $H=H\F_p$, $H_*(-)=H_*(-;\F_p)$. Then
$\TAQ$-Hurewicz homomorphism
\[
\theta'\:H_*(H)\to \TAQ_*(H,S;H)
\]
has the following effect on
\[
H_*(H) = \mathcal{A}(p)_* =
\begin{cases}
\F_p[\xi_i:i\geq1]\otimes\Lambda(\tau_j:j\geq0) & \text{if $p$ is odd}, \\
\F_2[\xi_i:i\geq1] & \text{if $p=2$}.
\end{cases}
\]
When $p$ is odd,
\[
\theta'(\tau_0) \neq 0,
\quad
\theta'(\tau_i) = \theta'(\xi_i) = 0 \quad (i\geq1).
\]
When $p=2$,
\[
\theta'(\xi_1) \neq0,
\quad
\theta'(\xi_i) = 0 \quad (i\geq2).
\]

The vanishing results follows from Steinberger's calculations
of Dyer-Lashof operations in~\cite{LNM1176}*{chapter~III, theorem~2.3}.
The non-triviality results use the fact that the unit $S\to H$
is $0$-connected, hence by Basterra~\cite{MBtaq}*{lemma~8.2},
$\Omega_S(H)$ is $0$-connected, see also~\cite{BGRtaq}*{corollary~1.3}.
\end{examp}

Next we will consider the case of $MO$. The infinite loop
space $BO$ has Thom spectrum $MO$ which admits the structure
of an \Einfty  ring spectrum or equivalently of a commutative
$S$-algebra. By Thom's theorem, this is known to split as a
wedge of suspensions of $H=H\F_2$ even as a ring spectrum
\[
MO \sim \bigvee_\alpha\Sigma^\alpha H.
\]
But as we will see, no such splitting can happen in
$\bar{h}\mathscr{C}_S$ because of obstructions lying in
$\TAQ$. Here the underlying infinite loop space is $BO$ and
the associated spectrum is~$ko\<1\>$, the $0$-connected cover
of~$ko$. In the above splitting, the generalized Eilenberg-Mac~Lane
ring spectrum on the right hand side realises the graded
polynomial ring
\begin{equation}\label{eqn:pi*MO}
MO_* = \pi_*(MO)
 = \F_2[z_n : \text{$n\geq 1$ is not of the form $2^s-1$}],
\end{equation}
where $z_n$ has degree $n$. For more on such ring spectra,
see~\cite{JMB:GEM}. Let $\underline{h}\:\pi_n(MO) \to H_n(MO)$
denote the usual mod~$2$ homology Hurewicz homomorphism.
By Thom's theorem, $\underline{h}$ is a monomorphism and
for the polynomial generators $z_n$ of~\eqref{eqn:pi*MO},
the Hurewicz images $\underline{h}(z_n)$ form part of a
set of polynomial generators for $H_*(MO)$ which has one
generator in each positive degree.

By a result of Basterra and Mandell~\cite{MB-MM:TAQ},
\[
\Omega_S(MO) = MO \wedge ko\<1\>,
\]
where $ko\<1\>$ is the $0$-connected cover of $ko$, defined
by the cofibre sequence of $ko$-modules
\[
ko\<1\> \to ko \to H\Z \to \Sigma ko\<1\>.
\]
On applying~$\bmod{\;2}$ homology $H_*(-)$ we obtain a short
exact sequence
\[
0\ra H_*(\Sigma^{-1} ko) \to H_*(\Sigma^{-1} H\Z) \to H_*(ko\<1\>) \ra0
\]
from which we deduce that as an $H_*ko$-module,
\begin{equation}\label{eqn:H*ko<1>}
H_*(ko\<1\>) =
H_*(ko)\{\Sigma^{-1}\zeta_1^2,\Sigma^{-1}\zeta_2,\Sigma^{-1}\zeta_1^2\zeta_2\},
\end{equation}
\ie, the free $H_*(ko)$-module on the generators
$\Sigma^{-1}\zeta_1^2,\Sigma^{-1}\zeta_2,\Sigma^{-1}\zeta_1^2\zeta_2$
which have degrees~$1,2,4$ respectively.

We will make use of the $\TAQ$-Hurewicz homomorphism
\[
\theta\: \pi_n(MO) \to \TAQ_n(MO,S;\F_2) = H_n(ko\<1\>),
\]
and so we need to understand the mod~$2$ homology $H_*(ko\<1\>)$.
In the dual Steenrod algebra
\[
\mathcal{A}(2)_* = H_*(H) = \F_2[\xi_r:r\geq1] = \F_2[\zeta_r:r\geq1],
\]
each generator $\xi_r\in\mathcal{A}(2)_{2^r-1}$ is in the image
of the natural map
\[
H_{2^r}(\RP^\infty) \to H_{2^r}(\Sigma H) = \mathcal{A}(2)_{2^r-1},
\]
and $\zeta_r=\chi(\xi_r)$, the Hopf-algebra conjugate of $\xi_r$.

Now since $\pi_1(ko\<1\>)=\F_2$, there is a canonical non-trivial
homotopy class $\psi\:ko\<1\>\to\Sigma H$ inducing an isomorphism
on $\pi_1(-)$. The horizontal composition in the diagram
\[
\xymatrix{
H\Z \ar[r]\ar@{..>}@/_0.8pc/[dr]_{\text{reduction $\bmod{\;2}$\;}}
          & \Sigma ko\<1\> \ar[r]^{\ph{a}\psi} & \Sigma^2H \\
& H\ar@{..>}@/_0.8pc/[ur]_{\Sq^2}&
}
\]
factors as shown. In order to calculate the effect of the
$H_*(ko)$-module homomorphism
\[
\psi_*\:H_*(ko\<1\>)\to H_{*-1}(H),
\]
we first note that for $r=1,2$, the composition
\[
ko\to H \xrightarrow{\Sq^r} \Sigma^2H
\]
is trivial, hence it induces the trivial map on $H_*(ko)$.
Using the Cartan formula for $\Sq^2_*$, for any element
$w\in H_*(ko)$ we obtain
\begin{equation}\label{eq:H*ko<1>}
\psi_*(w\Sigma^{-1}\zeta_1^2) = w, \quad
\psi_*(w\Sigma^{-1}\zeta_2) = w\zeta_1 = w\xi_1, \quad
\psi_*(w\Sigma^{-1}\zeta_1^2\zeta_2) = w(\zeta_2+\zeta_1^3) = w\xi_2.
\end{equation}
In particular it follows that $\psi_*\:H_*(ko\<1\>)\to H_{*-1}(H)$
is a monomorphism. We also note that the factorisation
of $\eta\:\Sigma ko\to ko$ through a $ko$-module map
$\tilde\eta\:\Sigma ko\to ko\<1\>$ induces
\[
\tilde\eta_*\:H_*(\Sigma ko) \to H_*(ko\<1\>);
\quad
\tilde\eta_*(w) = w\Sigma^{-1}\zeta_1^2.
\]

\begin{prop}\label{prop:MO-theta}
For any choice of generators $z_n$ in~\eqref{eqn:pi*MO},
the\/ $\TAQ$-Hurewicz homomorphism
$\theta\:\pi_*(MO)\to H_*(ko\<1\>)$ satisfies
\[
\theta(z_n) =
\begin{cases}
\ph{abc}0 & \text{\rm if $n\neq2^s$}, \\
\Sigma^{-1}\zeta_2 & \text{\rm if $n=2$}, \\
\Sigma^{-1}\zeta_1^2\zeta_2 & \text{\rm if $n=4$}, \\
\Sigma^{-1}\zeta_1^2\xi_s & \text{\rm if $n=2^s$ with $s\geq3$}. \\
\end{cases}
\]
Hence $MO$ is not a minimal atomic $2$-complete commutative
$S$-algebra.
\end{prop}
\begin{proof}
Choose polynomial generators $a_n\in H_n(MO)$ so that
when~$n+1$ is not a power of~$2$,
\[
\underline{h}(z_n) = a_n.
\]
Note that Kochman's results in~\cite{SOK:DLops} give the
action of the Dyer-Lashof operations on $H_*(BO)$ and the
Dyer-Lashof indecomposables are spanned by the polynomial
generators $a_{2^s}$ for $s\geq0$. Thus we should only
expect $\theta(z_n)$ to be non-zero when $n=2^s$ for some
$s\geq0$.

The calculation of $\psi_*\o\theta$ require similar methods
to those used for $MU$ in~\cite{BGRtaq}*{section~3}. The
crucial point is the determination of the homomorphism
\[
H_*(\RPi) \to H_*(BO) \to H_*(ko\langle1\rangle)
\]
induced by the natural inclusion $\RPi\to BO$ and the
evaluation
\[
\Sigma^\infty BO =
  \Sigma^\infty\Omega^\infty\Sigma^\infty ko\to ko.
\]
Composing with  $\psi$ and applying homology we obtain
\[
H_*(\RPi) \to H_*(ko\langle1\rangle)
\xrightarrow{\;\psi_*\;}
     H_*(\Sigma H) = \mathcal{A}(2)_{*-1},
\]
where $\psi_*$ is monic. Since $H^1(\RPi)=\F_2$, our
composition is the standard one which maps the generator
$\gamma_n\in H_n(\RPi)$ according to the rule
\[
\gamma_n \mapsto
\begin{cases}
\xi_s & \text{if $n=2^s-1$}, \\
0 &\text{otherwise}.
\end{cases}
\]
Using~\eqref{eq:H*ko<1>} we see that $\theta'(a_{2^s})$
has the form claimed.

The statement about $MO$ not being minimal atomic follows
from Thom's result since by definition, for each $s\geq1$,
$a_{2^s}$ is the Hurewicz image of a homotopy element.
\end{proof}

For completeness, we mention the following result which
appeared in the unpublished preprint of \Kriz~\cite{IK:BP},
unpublished work of Basterra and Mandell, and
Lazarev~\cite{AL:Glasgow}.
\begin{thm}\label{thm:Kriz-HF2}
There is an isomorphism
\[
\TAQ^*(H\F_2,S;H\F_2) \iso
\F_2\{\Sigma\SQ^I:
  \text{\rm $I=(i_1,\ldots,i_t)$ admissible, $i_t\geq4$}\}.
\]
\end{thm}
Here the symbols $\SQ^I$ behave like the analogous symbols
$\Sq^I$ in the Steenrod algebra $\mathcal{A}(2)^*$, and we
regard the empty sequence as admissible. However, the right
hand side should not be regarded as a module over the Steenrod
algebra $\mathcal{A}(2)^*$, and this is merely an isomorphism
of vector spaces. Here the suspension $\Sigma(-)$ indicates
a degree shift of~$+1$. There is a duality between
$\TAQ^*(H\F_2,S;H\F_2)$ and $\TAQ_*(H\F_2,S;H\F_2)$, i.e.,
\[
\TAQ^n(H\F_2,S;H\F_2) \iso \Hom_{\F_2}(\TAQ_n(H\F_2,S;H\F_2),\F_2).
\]
%%%In particular we see that
%%%\[
%%%\dim_{\F_2}\TAQ_5(H\F_2,S;H\F_2) = \dim_{\F_2}\TAQ^5(H\F_2,S;H\F_2) = 1,
%%%\]
%%%where $\TAQ^5(H\F_2,S;H\F_2)$ is spanned by $\Sigma\SQ^4$.

We end with a result that is essentially a generalisation
of~\cite{Hu-Kriz-May}*{proposition~2.11}; several examples of this
kind were given in Helen Gilmour's thesis~\cite{HG:PhD}.
%%%However here we use computations in $\TAQ$ rather than directly
%%%using Dyer-Lashof operations; the precise connections only
%%%become clear in the setting of the Miller-type spectral sequence
%%%which we will consider in a planned part~II of this work.
In the planned part~II of this work, we will describe computations
in the setting of a Miller-type spectral sequence for computing
$\TAQ$ which also lead to such results.
\begin{prop}\label{prop:H->MO}
There is no morphism of commutative $S$-algebras $H\F_2\to MO$.
\end{prop}
\begin{proof}
Once again we set $H=H\F_2$.
%%%
%%%By~\eqref{eqn:H*ko<1>},
%%%\[
%%%\TAQ_*(MO,S;H) =
%%%H_*(ko)\{\Sigma^{-1}\zeta_1^2,\Sigma^{-1}\zeta_2,\Sigma^{-1}\zeta_1^2\zeta_2\},
%%%\]
%%%hence $\TAQ_5(MO,S;H) = 0$. If such a morphism $H\F_2\to MO$
%%%exists, then since the composition
%%%\[
%%%\xymatrix{
%%%H \ar[r]\ar@/^13pt/[rr]^{\sim} & MO \ar[r] & H
%%%}
%%%\]
%%%is the identity on the bottom cell, it must be a weak equivalence
%%%since~$H$ is atomic~\cite{AJB-JPM}. Applying $\TAQ_*(-,S;H)$ we
%%%obtain a commutative diagram
%%%\[
%%%\xymatrix{
%%%\TAQ_*(H,S;H) \ar[r]\ar@/^15pt/[rr]^{\iso}
%%%                    & \TAQ_*(MO,S;H)\ar[r] & \TAQ_*(H,S;H)
%%%}
%%%\]
%%%where $\TAQ_6(H,S;H) \neq 0$ and $\TAQ_6(MO,S;H) = 0$, giving
%%%a contradiction.
%%%

If such a morphism $H\to MO$ existed, the generator
$\zeta_1\in H_1H=\mathcal{A}(2)_1$ would map to the
algebra generator $a_1\in H_1MO$. Using the Dyer-Lashof
action calculated by Kochman~\cite{SOK:DLops}, see
Theorem~\ref{thm:DL-MU}, we have
\[
\dlQ^4a_1 \equiv a_5 \mod{\text{decomposables}}.
\]
As in the proof of~\cite{Hu-Kriz-May}*{proposition~2.11},
this leads to a contradiction since there is no degree $5$
indecomposable in $\mathcal{A}(2)_*$.
\end{proof}

We will not give the details here, but it seems worth mentioning
that the Thom spectrum $MU/O$ associated to the infinite loop
space $U/O$ which is the fibre in the sequence
\[
U/O \to BO \to BU,
\]
is a core for $MO$. It turns out that $H_*(MU/O;\F_2)$ embeds
into $H_*(MO;\F_2)$ as a polynomial subalgebra on odd degree
generators the only Dyer-Lashof indecomposable has degree~$1$.
In fact
\[
\Omega_S(MU/O) \iso MU/O\wedge \Sigma ko
\]
and so
\[
\TAQ_*(MU/O,S;H\F_2) = H_*(\Sigma ko;\F_2),
\]
and under the $\TAQ$-Hurewicz homomorphism the Dyer-Lashof
indecomposable generator is sent to $\Sigma1$.

\section{\Einfty  orientations for complex line bundles}
\label{sec:UnivCplxOrient}

We end with a discussion involving the suspension spectrum
$\Sigma^{\infty-2}\CPi=\Sigma^{-2}\Sigma^{\infty}\CPi$
discussed in~\cite{AB-BR:Mxi}. A complex orientation 
for a ring spectrum $E$ is the homotopy class of a map 
$\Sigma^{\infty-2}\CPi\to E$ whose restriction
\[
S^0 \sim \Sigma^{-2}\Sigma^{\infty}\CP^1
    = \Sigma^{-2}\Sigma^{\infty}S^2\to E
\]
is homotopic to the unit of~$E$. When $E$ is a commutative
$S$-algebra, such a map $\Sigma^{\infty-2}\CPi\to E$ induces
a unique morphism of commutative $S$-algebras
\[
\mathbb{P}\Sigma^{\infty-2}\CPi \to E.
\]
Because of the condition involving the bottom cell, there
is a commutative diagram of solid arrows
\[
\xymatrix{
\mathbb{P}S^0
\ar@{}[dr]|{\PO}
%{\text{\Large$\ulcorner$}}
             \ar[r]\ar@{ >->}[d]
 & \mathbb{P}\Sigma^{\infty-2}\CPi\ar@{ >->}[d]\ar@/^21pt/[ddr] & \\
\mathbb{P}D^1\ar[r]\ar@/_21pt/[drr]
 & \tilde{\mathbb{P}}\Sigma^{\infty-2}\CPi\ar@{.>}[dr] & \\
 & & E
}
\]

\bigskip
\noindent
and hence a unique dotted arrow making the whole diagram
commute. This shows that $\tilde{\mathbb{P}}\Sigma^{\infty-2}\CPi$
is universal for maps to $E$ which give complex orientations
for complex line bundles. Of course the inclusion map
$\Sigma^{\infty-2}\CPi\to\tilde{\mathbb{P}}\Sigma^{\infty-2}\CPi$
itself provides a complex orientation.
%\end{examp}

\begin{lem}\label{lem:UnivCplxOrient}
The universal complex orientation
\[
\Sigma^{\infty-2}\CPi = \Sigma^{\infty-2}MU(1) \to MU
\]
induces a rational equivalence of commutative $S$-algebras
$\sigma\:\tilde{\mathbb{P}}\Sigma^{\infty-2}\CPi\to MU$.
Furthermore, the inclusion
$\Sigma^{\infty-2}\CPi\to\tilde{\mathbb{P}}\Sigma^{\infty-2}\CPi$
induces a morphism of ring spectra
\[
MU\to\tilde{\mathbb{P}}\Sigma^{\infty-2}\CPi
\]
which provides a splitting of $\sigma$ in the homotopy
category $\bar{h}\mathscr{M}_{MU}$.
\end{lem}
\begin{proof}
The rational result is straightforward since an argument
using the K\"unneth spectral sequence gives
\[
H_*(\tilde{\mathbb{P}}\Sigma^{\infty-2}\CPi;\Q)
 = \Q[\tilde{\beta}_r : r\geq1],
\]
where $\tilde{\beta}_r$ is the image of the canonical
generator $\beta_r\in H_{2r}(\CPi)$. Then the morphism
$\tilde{\mathbb{P}}\Sigma^{\infty-2}\CPi\to MU$ induces
an isomorphism of rings
\[
H_*(\tilde{\mathbb{P}}\Sigma^{\infty-2}\CPi;\Q)
                                       \to H_*(MU;\Q).
\]
It is easy to see that
\[
H_*(\tilde{\mathbb{P}}\Sigma^{\infty-2}\CPi;\Z)
                                       \to H_*(MU;\Z)
\]
is epic.

The composition
\[
\Sigma^{\infty-2}\CPi\to MU
        \to\tilde{\mathbb{P}}\Sigma^{\infty-2}\CPi
        \xrightarrow{\;\sigma\;} MU
\]
is homotopic to the canonical orientation, so the composition
\[
MU\to\tilde{\mathbb{P}}\Sigma^{\infty-2}\CPi
                                 \xrightarrow{\;\sigma\;} MU
\]
is homotopic to the identity by the classical universality of
the commutative ring spectrum~$MU$ described by
Adams~\cite{JFA:Blue}.
\end{proof}

We remark that
$\tilde{\mathbb{P}}\Sigma^{\infty-2}\CPi$ is weakly equivalent
to the Thom spectrum associated with the infinite loop map
$\Omega^{\infty}\Sigma^{\infty}\CPi\to BU$ extending the natural
map $\CPi\to BU$ whose fibre $F$ has torsion homotopy groups;
in fact, a result of Graeme Segal shows that there is an
equivalence of spaces
\[
\Omega^{\infty}\Sigma^{\infty}\CPi \sim  BU\times F.
\]

The morphism $\tilde{\mathbb{P}}\Sigma^{\infty-2}\CPi\to MU$
can be converted into a fibration (in either of the two model
categories $\mathscr{M}_S$ or $\mathscr{C}_S$), giving a
commutative diagram
\[
\xymatrix{
\tilde{\mathbb{P}}\Sigma^{\infty-2}\CPi \ar[dr]\ar@{ >->}[rr]^{\sim}
& &
(\tilde{\mathbb{P}}\Sigma^{\infty-2}\CPi)'\ar@{->>}[dl] \\
 & MU &
}
\]
where $(\tilde{\mathbb{P}}\Sigma^{\infty-2}\CPi)'$ is cofibrant
in $\mathscr{C}_S$. The map
$(\tilde{\mathbb{P}}\Sigma^{\infty-2}\CPi)'\to MU$ is a morphism
in the subcategory
$\mathscr{M}_{(\tilde{\mathbb{P}}\Sigma^{\infty-2}\CPi)'}$
of $\mathscr{M}_S$.
\begin{cor}\label{cor:UnivCplxOrient}
The fibre of\/ $(\tilde{\mathbb{P}}\Sigma^{\infty-2}\CPi)'\to MU$
is rationally trivial.
\end{cor}
A version of the next result appears in \cite{AB-BR:Mxi}.
\begin{prop}\label{prop:MU->PCP}
For a prime $p$, there can be no morphism of commutative
$S_{(p)}$-algebras
$\theta\:MU\to(\tilde{\mathbb{P}}\Sigma^{\infty-2}\CPi)_{(p)}$
for which $\sigma\circ\theta$ is a weak equivalence. Hence
there can be no morphism of commutative $S$-algebras
$\theta\:MU\to\tilde{\mathbb{P}}\Sigma^{\infty-2}\CPi$ for
which $\sigma\circ\theta$ is a weak equivalence.
\end{prop}
\begin{proof}
It suffices to prove the result for a prime~$p$, and we will
assume all spectra are localised at~$p$. Assume such a morphism
$\theta$ existed. Then by naturality of $\Omega_S$, there are
(derived) morphisms of $MU$-modules and a commutative diagram
\[
\xymatrix{
\Omega_{S}(MU) \ar[r]_(.4){\theta} \ar@/^19pt/[rr]^{\sim}
& \Omega_{S}(\tilde{\mathbb{P}}\Sigma^{\infty-2}\CPi)\ar[r]_(.6){\sigma}
& \Omega_{S}(MU)
}
\]
which induces a commutative diagram in $\TAQ_*(-;H\F_p)$
\[
\xymatrix{
H_*(\Sigma^2ku;\F_p) \ar[r]_(.46){\theta_*} \ar@/^19pt/[rr]^{\iso}
& H_*(\Sigma^{\infty-2}\CPi_2;\F_p)\ar[r]_(.56){\sigma_*}
& H_*(\Sigma^2ku;\F_p)
}
\]
where $\CPi_2 = \CPi/\CP^1$, and we use a result
due to Basterra \& Mandell~\cite{MB-MM:TAQ} (for
further details see~\cite{BGRtaq}*{sections~4~\&~5})
to identify
$\TAQ_*(MU,S;H\F_p)$, namely
\[
\Omega_S(MU) \sim MU\wedge\Sigma^2ku.
\]

It is standard that
\[
H_n(\Sigma^{\infty-2}\CPi_2;\F_p) =
\begin{cases}
\F_p & \text{if $n\geq2$ and is even}, \\
\;0 & \text{otherwise}.
\end{cases}
\]
On the other hand, when $p=2$,
\[
H_*(ku;\F_2) = \F_2[\zeta_1^2,\zeta_2^2,\zeta_3,\zeta_4,\ldots]
                        \subseteq\mathcal{A}(2)_*
\]
with $|\zeta_s|=2^s-1$, while when $p$ is odd,
$\Sigma^2ku\sim \bigvee_{1\leq r\leq p-1}\Sigma^{2r}\ell$
with
\[
H_*(\ell;\F_2) =
\F_p[\zeta_1,\zeta_2,\zeta_3,\ldots]\otimes\Lambda(\bar{\tau}_r:r\geq2)
\]
where $|\zeta_s|=2p^s-2$ and $|\bar{\tau}_s|=2p^s-1$. Clearly
this means that no such $\theta$ can exist.
\end{proof}

At the prime~$2$, $\Sigma^{\infty-2}\CPi$ is known to be minimal
atomic~\cite{AJB-JPM}*{proposition~5.9}. The next result shows
that the functor $\tilde{\mathbb{P}}$ need not preserve this
property; see Proposition~\ref{prop:MinAtom} for a converse
to this.
\begin{prop}\label{prop:P-minatom}
The $2$-local commutative $S$-algebra
$\tilde{\mathbb{P}}\Sigma^{\infty-2}\CPi_{(2)}$ is not minimal
atomic.
\end{prop}
\begin{proof}
If $\tilde{\mathbb{P}}\Sigma^{\infty-2}\CPi_{(2)}$ were minimal
atomic then by~\cite{BGRtaq}*{theorem~3.3}, the TAQ Hurewicz
homomorphism (induced from the universal derivation)
\[
\theta\:\pi_n(\tilde{\mathbb{P}}\Sigma^{\infty-2}\CPi_{(2)})
\to \TAQ_n(\tilde{\mathbb{P}}\Sigma^{\infty-2}\CPi_{(2)},S;H\F_2))
\]
would be trivial for $n>0$.

By naturality, there is a commutative diagram
\[
\xymatrix{
\pi_*(\tilde{\mathbb{P}}\Sigma^{\infty-2}\CPi_{(2)})
\ar@{->>}[rr]^{\sigma_*}\ar[d]_{\theta} && \pi_*(MU_{(2)})\ar[d]^{\theta} \\
\TAQ_n(\tilde{\mathbb{P}}\Sigma^{\infty-2}\CPi_{(2)},S;H\F_2)
\ar[rr]^{\sigma_*}\ar[d]_{\iso} & & \TAQ_n(MU_{(2)},S;H\F_2))\ar[d]^{\iso} \\
H_*(\CPi_2;\F_2)\ar[rr]^{\sigma_*} & & H_*(\Sigma^2ku;\F_2)
}
\]
in which the surjectivity of the top row follows from
Lemma~\ref{lem:UnivCplxOrient}. The $2$-primary calculations
of~\cite{BGRtaq}*{section~5} show that the right hand Hurewicz
homomorphism~$\theta$ is non-zero in positive degrees, hence
so is the left hand one. Therefore
$\tilde{\mathbb{P}}\Sigma^{\infty-2}\CPi_{(2)}$ cannot be
minimal atomic.
\end{proof}

We leave the interested reader to formulate and verify analogues
for an odd prime~$p$ based on desuspensions of the $p$-local
summands of $\Sigma^\infty\CPi_{(p)}$.

\begin{rem}\label{rem:TAQ-comparison}
We point out that $\sigma_*\: H_*(\CPi_2;\F_2)\to H_*(\Sigma^2ku;\F_2)$
is different from the homomorphism induced by any map of spectra
$\Sigma^{\infty-2}\CPi_2\to\Sigma^2ku$ which is an equivalence
on the bottom cell. For such a map composed with the natural map
$\Sigma^2ku\to H\F_2$ induces the homomorphism in homology given
by
\[
\Sigma^{-2}\beta_n \mapsto
\begin{cases}
\xi_s^4 &\text{if $n=2^s$}, \\
0       &\text{otherwise}.
\end{cases}
\]
But also $\Sigma^{-2}\beta_n \mapsto b_{n-1}$ in $H_*(MU;\F_2)$
and under the $\TAQ$-Hurewicz homomorphism,
\[
b_{n-1} \mapsto
\begin{cases}
\xi_s^2 &\text{if $n=2^s+1$}, \\
0       &\text{otherwise}.
\end{cases}
\]
\end{rem}

As promised above, here is a positive result relating the
additive and multiplicative notions of minimal atomic. We
regard $p$-local spectra as equivalent to $S$-modules.

\begin{prop}\label{prop:MinAtom}
Let $p$ be a prime and let $S$ be the $p$-local sphere
spectrum. Suppose that $X$ is a connective Hurewicz
$S$-module with chosen bottom cell $S^0\to X$. If\/
$\tilde{\mathbb{P}}X$ is minimal atomic as a commutative
$S$-algebra, then $X$ is minimal atomic as an $S$-module.
\end{prop}
\begin{proof}
Working in $S^0/\mathscr{M}_S$ we can replace $X$ by a
CW spectrum which is weakly equivalent to it so we will
assume that this has been done. Using observations in
Remark~\ref{rem:S^0/pushouts}, we can relate the two
$n$-skeleta. The $(n+1)$-skeleton $X^{[n+1]}$ is
constructed using a map of $S$-modules
\[
i^n\:\bigvee_iS^n \to X^{[n]}
\]
for which
\[
\ker[i^n_*\:\pi_n(\bigvee_iS^n) \to \pi_n(X)^{\<n\>}]
 \subseteq p\,\pi_n(\bigvee_iS^n).
\]
Similarly, we form the $(n+1)$-skeleton
$\tilde{\mathbb{P}}^{\<n+1\>}X$ is constructed from
$\tilde{\mathbb{P}}^{\<n\>}X$ using a morphism of
$S$-modules
\[
j^n\:\bigvee_iS^n \to \tilde{\mathbb{P}}^{\<n\>}X
\]
for which
\[
\ker[j^n_*\:\pi_n(\bigvee_iS^n) \to \pi_n(\tilde{\mathbb{P}}^{\<n\>}X)]
       \subseteq p\,\pi_n(\bigvee_iS^n).
\]
In $S^0/\mathscr{M}_S$ there is a commutative diagram
\[
\xymatrix{
    & \ar[dl]_{i^n}\bigvee_iS^n\ar[dr]^{j^n} & \\
X^{[n]}\ar[rr]^{\mathrm{incl}} & & \tilde{\mathbb{P}}^{\<n\>}X
}
\]
in which~$i^n$ provides the attaching maps for the $(n+1)$-cells
of~$X$. Clearly
\[
\ker[i^n_*\:\pi_n(\bigvee_iS^n) \to \pi_n(X^{[n]})]
\subseteq
\ker[j^n_*\:\pi_n(\bigvee_iS^n) \to \pi_n(\tilde{\mathbb{P}}^{\<n\>}X)]
        \subseteq p\,\pi_n(\bigvee_iS^n),
\]
and it follows that $X$ is nuclear.
\end{proof}

%%%%%%%%%%%%%%%%%%%%%%%

\appendix

\section{A proof and a Lemma}\label{sec:Missingpf}

For completeness we outline a proof of~\cite{BGRtaq}*{proposition~1.6},
due to Philipp Reinhard; unfortunately this was only produced
after that paper was published. Our approach is similar to that
of McCarthy and Minasian in~\cite{RM&VM}*{theorem~6.1}, however
this appears to be incorrect as stated (at one stage they seem
to assume that~$M$ is an algebra).

\begin{prop}\label{prop:TAQ-prop1.6}
Let $R$ be a commutative $S$-algebra and let $X$ be a cofibrant
$R$-module. Then there is a weak equivalence of
$\mathbb{P}_RX$-modules
\[
\Omega_R(\mathbb{P}_R X) \sim \mathbb{P}_R X \wedge_R X.
\]
\end{prop}

\begin{proof}
For every $M\in\mathscr{M}_{\mathbb{P}_R X}$ there is an adjunction
\[
\mathscr{C}_R / \mathbb{P}_R X (\mathbb{P}_R X,\mathbb{P}_R X\vee M)
 \iso \mathscr{M}_R / \mathbb{P}_R X(X,\mathbb{P}_R X\vee M),
\]
where $\mathscr{M}_R / \mathbb{P}_R X$ denotes the category of
$R$-modules over $\mathbb{P}_R X$. Because the forgetful functor
$\mathscr{C}_R / \mathbb{P}_R X \to\mathscr{M}_R/\mathbb{P}_R X$
respects fibrations and acyclic fibrations, the adjunction passes
to homotopy categories, giving
\[
\bar{h}\mathscr{C}_R/\mathbb{P}_R X(\mathbb{P}_R X,\mathbb{P}_R X\vee M)
\iso \bar{h}\mathscr{M}_R /\mathbb{P}_R X(X,\mathbb{P}_R X\vee M).
\]
Now we have
\[
\mathscr{M}_R /\mathbb{P}_R X (\mathbb{P}_R X, M)
                    \iso \mathscr{M}_R / X(X,X\vee M)
\]
and the adjunction again passes to homotopy categories
and gives
\[
\bar{h}\mathscr{M}_R /\mathbb{P}_R X (\mathbb{P}_R X, M)
\iso \bar{h}\mathscr{M}_R / X(X,X\vee M).
\]
Since in the homotopy category $X \vee M$ is the product
of~$X$ and~$M$, we have
\[
\bar{h}\mathscr{M}_R / X(X,X\vee M)\iso\bar{h}\mathscr{M}_R(X,M).
\]
By using the free functor from $R$-modules to $\mathbb{P}_R X$-modules,
we obtain
\[
\bar{h}\mathscr{M}_R / X(X,X\vee M) \iso
\bar{h}\mathscr{M}_{\mathbb{P}_R X}(\mathbb{P}_R X \wedge_R X, M).
\]
Thus we have shown that
\begin{align*}
\bar{h}\mathscr{M}_{\mathbb{P}_R X}(\Omega_R(\mathbb{P}_R X), M)
 &\iso \bar{h}\mathscr{C}_R / \mathbb{P}_R X(\mathbb{P}_R X,\mathbb{P}_R X \vee M) \\
 &\iso \bar{h} \mathscr{M}_{\mathbb{P}_R X}(\mathbb{P}_R X \wedge_R X, M).
\end{align*}
Using Yoneda's lemma, we obtain the desired equivalence
\[
\Omega_R(\mathbb{P}_R X)\sim \mathbb{P}_R X \wedge_R X.
\qedhere
\]
\end{proof}

We also give a useful result on the adjunction for a
commutative $R$-algebra. Let~$A$ be a cofibrant
commutative $R$-algebra and let
\[
\xymatrix{
A^c\ar@{->>}[r]^{\;\sim\;} & A
}
\]
be its functorial cofibrant replacement in the model
category of $R$-modules $\mathscr{M}_R$. Let
\[
\tilde{\mu}\:\mathbb{P}_RA^c\to\mathbb{P}_RA\to A
\]
be the extension of the multiplication. We have
\[
\Omega_R(\mathbb{P}_RA^c) \iso \mathbb{P}_RA^c \wedge_R A^c,
\]
and also the $A$-module $\Omega_R(A)$ becomes a
$\mathbb{P}_RA^c$-module via pullback along $\tilde{\mu}$.
Writing $\delta$ (without decorations) for universal
derivations, we obtain a commutative diagram in
$\bar{h}\mathscr{M}_{R}$ (with the pentagon commuting
in $\bar{h}\mathscr{M}_{\mathbb{P}_RA^c}$).
\begin{equation}\label{eq:Diffadjunct}
\xymatrix{
&R\wedge_R A^c\ar[rd]^{\mathrm{unit}}&& \\ %@/^15pt/
A^c\ar[r]\ar[rdd]_{\sim}\ar@{<->}[ur]^{\iso} & \mathbb{P}_RA^c\ar[dd]^{\tilde{\mu}}\ar[r]^(.4){\delta}
 & \mathbb{P}_RA^c\wedge_R A^c\ar[dr]^{\tilde{\mu}\wedge\delta}\ar[dd]^{\omega(\tilde{\mu})} & \\
 && & A\wedge_R\Omega_R(A)\ar[ld]^{\mathrm{mult}}   \\
 & A\ar[r]^(.4){\delta} & \Omega_R(A) &
}
\end{equation}
Here $\omega(\tilde{\mu})$ denotes the induced `derivative'
morphism $\Omega_R(\mathbb{P}_RA^c)\to\Omega_R(A)$.

\begin{lem}\label{lem:Diffadjunct}
Suppose that $M$ is an $A$-module and therefore an
$\mathbb{P}_RA^c$-module. Then the induced morphism on
$\TAQ_*(-)$, $\tilde{\mu}_*$, is given by the following
commutative diagram.
\[
\xymatrix{
\TAQ_*(\mathbb{P}_RA^c,R;M)\ar@{=}[d] \ar[rr]^{\tilde{\mu}_*}
                            && \TAQ_*(A,R;M)\ar@{=}[d] \\
 \pi_*(M\wedge_R A^c)\ar[dr]_{(I\wedge\delta_{(A,R)})_*}
    && \pi_*(M\wedge_A\Omega_R(A)) \\
    & \pi_*(M\wedge_R\Omega_R(A))\ar[ur] &
}
\]
\end{lem}
\begin{proof}
This is obtained by applying $\pi_*(M\wedge_R-)$ to~\eqref{eq:Diffadjunct}.
\end{proof}

Since the universal derivation restricts trivially to~$R$,
there is an induced map
\[
\bar{\delta}_{(A,R)}\:A^c/S_R^0\to\Omega_R(A).
\]
So for the reduced free algebra there is a similar commutative
diagram.
\[
\xymatrix{
\TAQ_*(\tilde{\mathbb{P}}_RA^c,R;M)\ar@{=}[d] \ar[rr]^{\tilde{\mu}_*}
                            && \TAQ_*(A,R;M)\ar@{=}[d] \\
 \pi_*(M\wedge_R A^c/S_R^0)\ar[dr]_{(I\wedge\bar{\delta}_{(A,R)})_*}
    && \pi_*(M\wedge_A\Omega_R(A)) \\
    & \pi_*(M\wedge_R\Omega_R(A))\ar[ur] &
}
\]

\section{Some formulae}\label{sec:Formulae}

We begin by recalling formula due to Kochman~\cite{SOK:DLops}.
Actually his results are for the infinite loop spaces such
as $BU$, but the Thom isomorphism commutes with the Dyer-Lashof
operations so we will interpret them in the homology of the
Thom spectrum $MU$ with its \Einfty  structure inherited
from that of $BU$.

Let $p$ be a prime. We will write $H_*(-)=H_*(-;\F_p)$. Let
$b_r\in H_{2r}(MU)$ be the generator obtained as the image
of $\beta_{r+1}\in H_{2r+2}(MU(1))\iso H_{2r+2}(\CPi)$
under the homomorphism induced by the canonical map
$MU(1)\to \Sigma^2MU$ as in~\cite{JFA:Blue}. We will use
the notation $x\approx y$ as shorthand for
$x\equiv y\bmod{\text{decomposables}}$. We also interchangeably
use the notations
\[
(a,b) = (b,a) = \binom{a+b}{a} = \binom{a+b}{b}
\]
for binomial coefficients, where this is taken to be zero
if $a<0$ or $b<0$. We will use the well-known congruence
\begin{equation}\label{eq:Binom-modp}
\binom{n_0+n_1p+\cdots +n_kp^k}{m_0+m_1p+\cdots +m_kp^k}
\equiv
\binom{n_0}{m_0}\binom{n_0}{m_0}\cdots \binom{n_k}{m_k}
                                       \mod{p}
\end{equation}
when $0\leq m_i,n_i\leq p-1$.
\begin{thm}\label{thm:DL-MU}
In $H_*(MU)$ we have
\begin{itemize}
\item
if $p$ is odd,
\[
\dlQ^rb_n \approx (-1)^{r+n+1}(n,r-n-1)b_{n+r(p-1)},
\]
\item
if $p=2$,
\[
\dlQ^{2r}b_n \approx (n,r-n-1)b_{n+r}.
\]
\end{itemize}
\end{thm}

Note that in the $p=2$ case there are analogous results
for $H_*(MO;\F_2)$.

The Dyer-Lashof operations annihilate $1$ and the Cartan
formula implies that they act on the indecomposable
quotient. In~\cite{SOK:DLops}*{theorem~10}, Kochman
determined the indecomposable generators which are not
in the image of any Dyer-Lashof operations of positive
degree. We set
\[
\mathrm{Q}_{\mathrm{DL}}H_*(MU) =
\mathrm{Q}H_*(MU)/\{\dlQ^sx:s\geq1,\;x\in\mathrm{Q}H_*(MU)\}.
\]
\begin{thm}\label{thm:DL-MU-QH}
The indecomposables $\mathrm{Q}_{\mathrm{DL}}H_*(MU)$
have the following elements as a basis:
\begin{itemize}
\item
if $p$ is odd, $b_{np^t}$ where $p\nmid n$ and
$n=(\sum_{i=0}^ks_ip^i)(p-1)+r$ with $r=1,2,\ldots,p-1$
and if $\sum_{i=0}^ks_ip^i\neq0$, $0\leq s_i\leq(p-1)$
and $1\leq s_k\leq s_{k-1}\leq\cdots\leq s_0\leq r$,
\item
if $p=2$, $b_{2^t}$ where $t\geq0$.
\end{itemize}
\end{thm}

The Dyer-Lashof indecomposability of the stated generators
can be deduced from our results on the $\TAQ$-Hurewicz
homomorphism. As an exercise in computing with binomial
coefficients modulo a prime, we have
\begin{prop}\label{prop:DL-decomp}
Suppose that $p$ is a prime and $n$ has $p$-adic expansion
\[
n = n_sp^s+\cdots n_{s+t}p^{s+t}
\]
where $n_s\neq0\neq n_{s+t}$ and $t>0$. \\
If $p$ is odd then
\[
\dlQ^{n-n_sp^s} b_{n_sp^s} \approx \pm\binom{n-n_sp^s-1}{n_sp^s} b_n
                        \not\approx 0.
\]
If $p=2$ then
\[
\dlQ^{2n-2^{s+1}} b_{2^s} \approx \binom{n-2^s-1}{2^s} b_n
                        \not\approx 0.
\]
\end{prop}
\begin{proof}
In each case working modulo~$p$ we have
\begin{align*}
\binom{n-n_sp^s-1}{n_sp^s} &\equiv \binom{(n-n_sp^s-p^k) + (p^k-1)}{n_sp^s} \\
                    &\equiv \binom{n-n_sp^s-p^k}{0}\binom{p^k-1}{n_sp^s} \\
                    &\not\equiv 0,
\end{align*}
where $p^k$ is the highest power of $p$ dividing $(n-n_sp^s)$,
and we use the fact that
\[
p^k-1 = (p-1)p^{k-1} +\cdots+(p-1)p^s+\cdots(p-1)p+(p-1)
\]
with $n_s\leq p-1$.
\end{proof}

\end{document}